\documentclass[a4paper, 11pt, final]{article}

\usepackage[notref,notcite]{showkeys}
\usepackage{amsfonts,amsbsy}
\usepackage{enumerate}
\usepackage{graphicx,psfrag}

\newtheorem{proposition}{Proposition}[section]
\newtheorem{theorem}[proposition]{Theorem}
\newtheorem{lemma}[proposition]{Lemma}
\newtheorem{corollary}[proposition]{Corollary}

\newtheorem{remark}[proposition]{Remark}

\newenvironment{proof}{\smallskip\noindent\emph{Proof.}\hspace{1pt}}%
{\hspace{-5pt}{\nobreak\quad\nobreak\hfill\nobreak$\square$\vspace{8pt}%
\par}\smallskip\goodbreak}

\newenvironment{proofof}[1]{\smallskip\noindent\emph{Proof of #1.}%
\hspace{1pt}}{\hspace{-5pt}{\nobreak\quad\nobreak\hfill\nobreak%
$\square$\vspace{8pt}\par}\smallskip\goodbreak}

\newcommand{\Section}[1]{\section{#1}\setcounter{equation}{0}}

\newcommand{\C}[1]{\mathbf{C^{#1}}}

\newcommand{\modulo}[1]{{\left|#1\right|}}
\newcommand{\norma}[1]{{\left\|#1\right\|}}
\newcommand{\Ref}[1]{{\rm(\ref{#1})}}
\newcommand{\reali}{{\mathbb{R}}}

\newcommand{\naturali}{{\mathbb{N}}}
\newcommand{\tv}{\mathrm{TV}}

\newcommand{\diam}{\mathop{\mathrm{diam}}}

\renewcommand{\epsilon}{\varepsilon}
\renewcommand{\phi}{\varphi}
\renewcommand{\theta}{\vartheta}
\renewcommand{\L}[1]{\mathbf{L^#1}}

\newcommand{\gf}{{\mathinner{\mathbf{\Phi}}}}

\newcommand{\caratt}[1]{{\chi_{\strut#1}}}

\title{On the Stability Functional for Conservation Laws}

\author{Rinaldo M.~Colombo \\ \small Dipartimento di Matematica \\
  \small Universit\`a degli Studi di Brescia \\ \small Via Branze, 38
  \\ \small 25123 Brescia, Italy \\\small
  \texttt{Rinaldo.Colombo@UniBs.it} \\ \and Graziano Guerra \\ \small
  Dip.~di Matematica e Applicazioni \\ \small Universit\`a di Milano
  -- Bicocca \\ \small Via Bicocca degli Arcimboldi, 8 \\
  \small 20126 Milano, Italy \\ \small
  \texttt{Graziano.Guerra@UniMiB.it}}

\begin{document}

\maketitle

\begin{abstract}

  \noindent This note is devoted to the explicit construction of a
  functional defined on all pairs of $\L1$ functions with small total
  variation, which is equivalent to the $\L1$ distance and non
  increasing along the trajectories of a given system of conservation
  laws. Two different constructions are provided, yielding an
  extension of the original stability functional by Bressan, Liu and
  Yang.

  \medskip

  \noindent\textit{2000~Mathematics Subject Classification:} 35L65.

  \medskip


\end{abstract}

\Section{Introduction}
\label{sec:Intro}

Let the smooth map $f \colon \Omega \mapsto \reali^n$ define the
strictly hyperbolic system of conservation laws
\begin{equation}
  \label{eq:HCL}
  \partial_t u + \partial_x f(u) = 0 
\end{equation}
where $t > 0$, $x \in \reali$ and $u \in \Omega$, with $\Omega
\subseteq \reali^n$ being an open set.

Most functional theoretic methods fail to tackle these equations,
essentially due to the appearance of shock waves. Since 1965, the
Glimm functional~\cite{Glimm} has been a major tool in any existence
proof for~\Ref{eq:HCL} and related equations. More recently, an
analogous role in the proofs of continuous dependence has been played
by the stability functional $\Phi$ introduced in~\cite{BressanYangLiu,
  LiuYang1, LiuYang3}, see also~\cite{BressanLectureNotes}. The
functional $\Phi$ has been widely used to prove the $\L1$--Lipschitz
dependence of solutions to~\Ref{eq:HCL} (and related problems) from
initial data having small total variation, see for
example~\cite{AmadoriGosseGuerra, AmadoriGuerra2002, ColomboGuerra1,
  ColomboGuerra2, Ha2, Ha1}.  Special cases comprising data with large
total variation are considered in~\cite{ColomboCorli3, Guerra2004,
  Lewicka2, Lewicka1, LewickaTrivisa}.  Nevertheless, the use of
$\Phi$ is hindered by the necessity of introducing specific
approximate solutions, namely the ones based either on Glimm
scheme~\cite{Glimm} or on the wave front tracking
algorithm~\cite{BressanLectureNotes, DafermosWFT}. The present paper
makes the use of the stability functional $\Phi$ \emph{independent}
from any kind of approximate solutions. The present construction
allows to simplify several parts of the cited papers, where the
presentation of the stability functional needs to be preceded by the
introduction of all the machinery related to Glimm's scheme or wave
front tracking approximations, see for instance~\cite{ColomboGuerra3}.

We extend the stability functional to all $\L1$ functions with
sufficiently small total variation. This construction is achieved in
two different ways. First, we use general piecewise constant functions
and a limiting procedure, without resorting to any sort of approximate
solutions. Secondly, we exploit the \emph{wave measures},
see~\cite[Chapter~10]{BressanLectureNotes} and give an equivalent
definition that does not require any limiting procedure. Furthermore,
we prove its lower semicontinuity.

With reference to~\cite{BressanLectureNotes} for the basic definitions
related to~\Ref{eq:HCL}, we state the main result of this paper.

\begin{theorem}
  \label{thm:one}
  Let $f$ generate a Standard Riemann Semigroup $S$ on the domain
  $\mathcal{D}_\delta$ defined at~\Ref{def:2.6}. Then, the functional 
  $\mathbf{\Xi}$ defined
  at~\Ref{eq:Xi} by means of piecewise constant functions coincides
  with the one defined at~\Ref{eq:barPhi} by means of wave measures.
  Moreover, it enjoys the following properties:
  \begin{enumerate}[(i)]
  \item $\mathbf{\Xi}$ is equivalent to the $\L1$ distance, i.e.~there
    exists a $C>0$ such that for all $u,\tilde u \in
    \mathcal{D}_{\delta}$
    \begin{displaymath}
      \frac{1}{C} \cdot \norma{u-\tilde u}_{\L1}
      \leq
      \mathbf{\Xi}(u,\tilde u)
      \leq 
      C \cdot \norma{u-\tilde u}_{\L1} \,.
    \end{displaymath}
  \item $\mathbf{\Xi}$ is non increasing along the semigroup
    trajectories, i.e.~for all $u,\tilde u \in\mathcal{D}_{\delta}$
    and for all $t \geq 0$
    \begin{displaymath}
      \mathbf{\Xi}(S_t u, S_t \tilde u) \leq \mathbf{\Xi} (u,\tilde u) \,.
    \end{displaymath}
  \item $\mathbf{\Xi}$ is lower semicontinuous with respect to the
    $\L1$ norm.
  \end{enumerate}
\end{theorem}

\noindent Taking advantage of the machinery presented below, we also
extend the classical Glimm functionals~\cite{BressanLectureNotes,
  Glimm} to general $\L1$ functions with small total variation and
prove their lower semicontinuity, recovering some of the results
in~\cite{BaitiBressan}, but with a shorter proof.

\Section{Notation and Preliminary Results}
\label{sec:Main}

Our reference for the basic definitions related to systems of
conservation laws is~\cite{BressanLectureNotes}. We assume throughout
that $0 \in \Omega$, with $\Omega$ open, and that
\begin{description}
\item[(F)] $f \in \C4 ( \Omega;\reali^n)$, the system~\Ref{eq:HCL} is strictly hyperbolic with
  each characteristic field either genuinely nonlinear or linearly
  degenerate.
\end{description}
\noindent Let $\lambda_1(u), \ldots, \lambda_n(u)$ be the $n$ real
distinct eigenvalues of $D\!f(u)$, indexed so that $\lambda_j (u) <
\lambda_{j+1}(u)$ for all $j$ and $u$. The $j$-th right eigenvector,
normalized as in~\Ref{eq:parametrization}--\Ref{eq:parametrizationBis},
is $r_j(u)$. Let $\sigma \mapsto R_j(\sigma)(u)$, respectively $\sigma
\mapsto S_j(\sigma)(u)$, be the rarefaction curve, respectively the
shock curve, exiting $u$, so that
\begin{equation}
  \label{eq:parametrization}
  \frac{\partial R_j(\sigma)(u)}{\partial \sigma} 
  = 
  r_j\left(R_j(\sigma)(u)\right) \,.
\end{equation}
If the $j$-th field is linearly degenerate, then the parameter
$\sigma$ above is the arc-length. In the genuinely nonlinear case,
see~\cite[Definition~5.2]{BressanLectureNotes}, we choose $\sigma$ so
that
\begin{equation}
  \label{eq:parametrizationBis}
  \frac{\partial \lambda_j}{\partial\sigma}
  \left(R_j (\sigma)(u) \right) 
  =
  k_j
  \quad , \quad
  \frac{\partial \lambda_j}{\partial\sigma}
  \left(S_j (\sigma)(u) \right) 
  =
  k_j 
\end{equation}
where $k_1,\ldots,k_n$ are arbitrary positive fixed numbers.
In~\cite{BressanLectureNotes} the choice $k_j=1$ for all
$j=1,\ldots,n$ was used, while in~\cite{AmadoriGuerra2002} another
choice was made to cope with diagonal dominant sources. The
choice~\Ref{eq:parametrizationBis} preserves the properties underlined
in~\cite[Remark~5.4]{BressanLectureNotes} so that the estimates
in~\cite[Chapter~8]{BressanLectureNotes} still hold. Introduce the
$j$-Lax curve
\begin{displaymath}
  \sigma \mapsto \psi_j (\sigma) (u) =
  \left\{
    \begin{array}{c@{\qquad\mbox{ if }\quad}rcl}
      R_j(\sigma)(u) & \sigma & \geq & 0
      \\
      S_j(\sigma)(u) & \sigma & < & 0
    \end{array}
  \right.
\end{displaymath}
and for $\boldsymbol{\sigma} \equiv (\sigma_1, \ldots, \sigma_n)$,
define the map
\begin{displaymath}
  \mathbf{\Psi}(\boldsymbol{\sigma})(u^-)
  =
  \psi_n(\sigma_n)\circ\ldots\circ\psi_1(\sigma_1)(u^-) \,.
\end{displaymath}
By~\cite[\S~5.3]{BressanLectureNotes}, given any two states $u^-,u^+
\in \Omega$ sufficiently close to $0$, there exists a map $E$ such
that
\begin{equation}
  \label{eq:E}
  (\sigma_1, \ldots, \sigma_n) = E(u^-,u^+)
  \quad \mbox{ if and only if } \quad
  u^+ = \mathbf{\Psi}(\boldsymbol{\sigma})(u^-) \,.
\end{equation}
Similarly, let $\mathbf{q} \equiv §(q_1, \ldots, q_n)$ and define the
map $\mathbf{S}$ by
\begin{equation}
  \label{eq:S}
  \mathbf{S}(\mathbf{q})(u^-) 
  =
  S_n(q_n) \circ \ldots \circ S_1(q_1) (u^-)
\end{equation}
as the gluing of the Rankine--Hugoniot curves. For any two states
$u^-, u^+$ as above, there exists a unique $\mathbf{q}$ such that $u^+
= \mathbf{S}(\mathbf{q})(u^-)$.

Let $u$ be piecewise constant with finitely many jumps and assume that
$\tv(u)$ is sufficiently small. Call $\mathcal{I}(u)$ the finite set
of points where $u$ has a jump.  Let $\sigma_{x,i}$ be the strength of
the $i$-th wave in the solution of the Riemann problem
for~\Ref{eq:HCL} with data $u(x-)$ and $u(x+)$, i.e.~$(\sigma_{x,1},
\ldots, \sigma_{x,n}) = E\left( u(x-), u(x +) \right)$. Obviously if
$x\not\in \mathcal{I}(u)$ then $\sigma_{x,i}=0$, for all
$i = 1,\ldots,n$.  As usual, $\mathcal{A}(u)$ denotes the set of
approaching waves in $u$:
\begin{displaymath}
  \mathcal{A} (u) 
  = \left\{
    \begin{array}{c}
      \left(
        (x,i),(y,j)\right) \in \left( \mathcal{I}(u) \times \{1,\ldots,n\} 
      \right)^2
      \colon 
      \\
      x < y \mbox{ and either } i > j \mbox{ or } i = j, 
      \mbox{ the $i$-th field}
      \\
      \mbox{is genuinely non linear, }
      \min \left\{ \sigma_{x,i}, \sigma_{y,j} \right\} <0\!  
    \end{array}
  \right\} \,.
\end{displaymath}
As in~\cite{Glimm} or~\cite[formula~(7.99)]{BressanLectureNotes},
the linear and the interaction potential are
\begin{displaymath}
  \mathbf{V}(u)
  =
  \sum_{x\in I(u)} \sum_{i=1}^n
  \modulo{\sigma_{x,i}}
  \quad \mbox{ and } \quad
  \mathbf{Q}(u)
  =
  \sum_{\left((x,i),(y,j)\right) \in \mathcal{A}(u)}
  \modulo{\sigma_{x,i}\sigma_{y,j}} \,.
\end{displaymath}
Moreover, let
\begin{equation}
  \label{def:ups}
  \mathbf{\Upsilon} (u) 
  = 
  \mathbf{V} (u) + C_0 \cdot \mathbf{Q} (u)  
\end{equation}
where $C_0>0$ is the constant appearing in the functional of the
wave--front tracking algorithm,
see~\cite[Proposition~7.1]{BressanLectureNotes}. Recall that $C_0$
depends only on the flow $f$ and on the upper bound of the total
variation of initial data.

\begin{remark}
  \label{rem:Lipschitz}
  For fixed $x_1< \ldots < x_{N+1}$, the maps
  \begin{displaymath}
    \begin{array}{rcl}
      (u_1, \ldots, u_N) 
      & \mapsto & 
      \mathbf{V} \left( 
        \sum_{\alpha=1}^N u_\alpha \, 
        \chi_{\left[x_{\alpha}, x_{\alpha+1}\right[}
      \right)
      \\
      (u_1, \ldots, u_N) 
      & \mapsto &
      \mathbf{Q} \left( 
        \sum_{\alpha=1}^N u_\alpha \, 
        \chi_{\left[x_{\alpha}, x_{\alpha+1}\right[}
      \right)
    \end{array}
  \end{displaymath}
  are Lipschitz continuous.  Moreover, the Lipschitz constant of the
  maps
  \begin{displaymath}
    u_{\bar\alpha}
    \mapsto 
    \mathbf{V} \left( 
      \sum_{\alpha=1}^N u_\alpha \, 
      \chi_{\left[x_{\alpha}, x_{\alpha+1}\right[}
    \right)
    \qquad
    u_{\bar\alpha}
    \mapsto 
    \mathbf{Q} \left( 
      \sum_{\alpha=1}^N u_\alpha \, 
      \chi_{\left[x_{\alpha}, x_{\alpha+1}\right[}
    \right)
  \end{displaymath}
  is bounded uniformly in $N$, $\bar\alpha$ and $u_\alpha$ for $\alpha
  \neq \bar \alpha$.
\end{remark}

Finally, for $\delta>0$ sufficiently small, we define
\begin{equation}
  \label{def:2.6}
  \begin{array}{c}
    \displaystyle
    \mathcal{D}_\delta^*
    =
    \left\{
      v \in \L1\left(\reali,\Omega \right) \colon
      v \hbox { is piecewise constant and } \mathbf{\Upsilon}(v) < \delta
    \right\}
    \\
    \displaystyle
    \mathcal{D}_\delta=
    \mathrm{cl} \, \mathcal{D}_\delta^*
  \end{array}
\end{equation}
where the closure is in the strong $\L1$--topology. Unless otherwise
stated, we always consider the right continuous representatives of
maps in $\mathcal{D}_\delta$ and $\mathcal{D}_\delta^*$.

For later use, for $u \in \mathcal{D}_{\delta}$ and $\eta >0$,
introduce the set
\begin{equation}
  \label{eq:Beta}
  B_\eta (u)
  = 
  \left\{
    v \in \L1(\reali;\Omega) \colon
    v \in\mathcal{D}_\delta^*
    \mbox{ and }
    \norma{v-u}_{\L1} < \eta
  \right\} \,.
\end{equation}
Note that, by the definition of $\mathcal{D}_\delta$, $B_\eta (u)$ is
not empty and if $\eta_1 < \eta_2$, then $B_{\eta_1}(u) \subseteq
B_{\eta_2}(u)$. Recall the following fundamental result, proved
in~\cite{BressanCrastaPiccoli}:

\begin{theorem}
  \label{thm:SRS}
  Let $f$ satisfy~\textbf{(F)}. Then, there exists a positive
  $\delta_o$ such that the equation~\Ref{eq:HCL} generates for all
  $\delta \in \left]0, \delta_o \right[$ a Standard Riemann Semigroup
  (SRS) $S\colon \left[0, +\infty \right[ \times \mathcal{D}_\delta
  \mapsto \mathcal{D}_\delta$, which is Lipschitz in $\L1$.
\end{theorem}

We refer to~\cite[Chapters~7 and~8]{BressanLectureNotes} for the proof
of the above result as well as for the definition and further
properties of the SRS.

\Section{The Piecewise Constant Functions Approach}
\label{sec:Glimm}

Extend the Glimm functionals to all $u \in \mathcal{D}_\delta$ as
follows:
\begin{equation}
  \label{eq:Q}
  \bar{\mathbf{Q}} (u) 
  =
  \lim_{\eta\to 0+} \inf_{v \in B_\eta(u)} \mathbf{Q}(v)
  \quad \mbox{ and } \quad
  \bar{\mathbf{\Upsilon}} (u) 
  =
  \lim_{\eta\to 0+} \inf_{v \in B_\eta(u)} \mathbf{\Upsilon}(v) \,.
\end{equation}
The maps $\eta \to \inf_{v \in B_\eta(v)} \mathbf{Q}(v)$ and $\eta \to
\inf_{v \in B_\eta(v)} \mathbf{\Upsilon}(v)$ are non increasing. Thus
the limits above exist and
\begin{displaymath}
  \bar{\mathbf{Q}} (u) 
  =
  \sup_{\eta > 0} \inf_{v \in B_\eta(u)} \mathbf{Q}(v)
  \quad \mbox{ and } \quad
  \bar{\mathbf{\Upsilon}} (u) 
  =
  \sup_{\eta > 0} \inf_{v \in B_\eta(u)} \mathbf{\Upsilon}(v) \,.
\end{displaymath}
We prove in Proposition~\ref{prop:coincide} below that
$\bar{\mathbf{Q}}$, respectively $\bar{\mathbf{\Upsilon}}$, coincides
with $\mathbf{Q}$, respectively $\mathbf{\Upsilon}$, when evaluated on
piecewise constant functions. Moreover, we prove in
Corollary~\ref{prop:equal} that $\bar{\mathbf{Q}}$ also coincides with
the functional intended in~\cite[formula~(1.15)]{BressanColombo2}
or~\cite[formula~(10.10)]{BressanLectureNotes}.  Preliminarily, we
exploit the formulation~\Ref{eq:Q} to prove directly the lower
semicontinuity of $\mathbf{Q}$ and $\mathbf{\Upsilon}$.

\begin{proposition}
  \label{prop:Qlsc}
  The functionals $\bar{\mathbf{Q}}$ and $\bar{\mathbf{\Upsilon}}$ are
  lower semicontinuous with respect to the $\L1$ norm.
\end{proposition}

\begin{proof}
  We prove the lower semicontinuity of $\bar{\mathbf{\Upsilon}}$, the
  case of $\bar{\mathbf{Q}}$ is analogous.

  Fix $u$ in $\mathcal{D}_\delta$. Let $u_\nu$ be a sequence in
  $\mathcal{D}_\delta$ converging to $u$ in $\L1$. Define
  $\epsilon_\nu = \norma{u_\nu-u}_{\L1} + 1/\nu$. Fix $v_\nu \in
  B_{\epsilon_\nu} (u_\nu)$ so that
  \begin{displaymath}
    \mathbf{\Upsilon} (v_\nu)
    \leq
    \inf_{v \in B_{\epsilon_\nu}(u_\nu)} \mathbf{\Upsilon}(v) +\epsilon_\nu
    \leq
    \bar{\mathbf{\Upsilon}} (u_\nu) + \epsilon_\nu \,.
  \end{displaymath}
  From $\norma{v_\nu - u}_{\L1} \leq \norma{v_\nu - u_\nu}_{\L1} +
  \norma{u_\nu - u}_{\L1} < 2 \epsilon_\nu$ we deduce $v_\nu \in
  B_{2\epsilon_\nu}(u)$ and the proof is completed with the following estimates:
  \begin{displaymath}
    \begin{array}{rcccl}
      \displaystyle
      \inf_{v \in  B_{2\epsilon_\nu}(u)} \mathbf{\Upsilon} (v) 
      & \leq &
      \displaystyle
      \mathbf{\Upsilon} (v_\nu) 
      & \leq &
      \displaystyle
      \bar{\mathbf{\Upsilon}} (u_\nu) + \epsilon_\nu \,;
      \\
      \displaystyle
      \bar{\mathbf{\Upsilon}} (u)
      & = &
      \displaystyle
      \lim_{\nu \to +\infty}
      \inf_{v \in  B_{2\epsilon_\nu}(u)} \mathbf{\Upsilon} (v) 
      & \leq &
      \displaystyle
      \liminf_{\nu\to+\infty} \bar{\mathbf{\Upsilon}} (u_\nu) \,.
    \end{array}
    \vspace{-\baselineskip}
  \end{displaymath}
\end{proof}

The next proposition contains in essence the reason why the Glimm
functionals $\mathbf{Q}$ and $\mathbf{\Upsilon}$ decrease. Compute
them on a piecewise constant function $u$ and \emph{``remove''} one
(or more) of the values attained by $u$, then the values of both
$\mathbf{Q}$ and $\mathbf{\Upsilon}$ decrease.

Let $u = \sum_{\alpha\in I} u_\alpha \,
\chi_{\left[x_{\alpha},x_{\alpha+1}\right[}$ be a piecewise constant
function, with $u_\alpha \in \Omega$, $x_1 < x_2 < \ldots < x_{N+1}$
and $I$ be a finite set of integers.  Then, we say that $u_1, u_2,
\ldots, u_N$ is the \emph{ordered sequence} of the values attained by
$u$ and we denote it by $(u_\alpha \colon \alpha \in I)$.

\begin{proposition}
  \label{prop:reduce}
  Let $u, \check u \in \mathcal{D}^*_\delta$. If the ordered sequence
  of the values attained by $u$ is $(u_\alpha \colon \alpha \in I)$
  and the ordered sequence of the values attained by $\check u$ is
  $(u_\alpha \colon \alpha \in J)$, with $J \subseteq I$, then
  $\mathbf{Q} (\check u) \leq \mathbf{Q}(u)$ and $\mathbf{\Upsilon}
  (\check u) \leq \mathbf{\Upsilon}(u)$.
\end{proposition}

\begin{proof}
  Consider the case $\sharp I = \sharp J +1$, see
  also~\cite[Lemma~10.2, Step~1]{BressanLectureNotes}.
  \begin{figure}
    \centering
    \begin{psfrags}
      \psfrag{m}{$u_{\bar\alpha-1}$} \psfrag{b}{$u_{\bar\alpha}$}
      \psfrag{p}{$u_{\bar\alpha+1}$} \psfrag{x}{$x$} \psfrag{u}{$u$}
      \psfrag{uc}{$\check u$}
      \includegraphics[width=11cm]{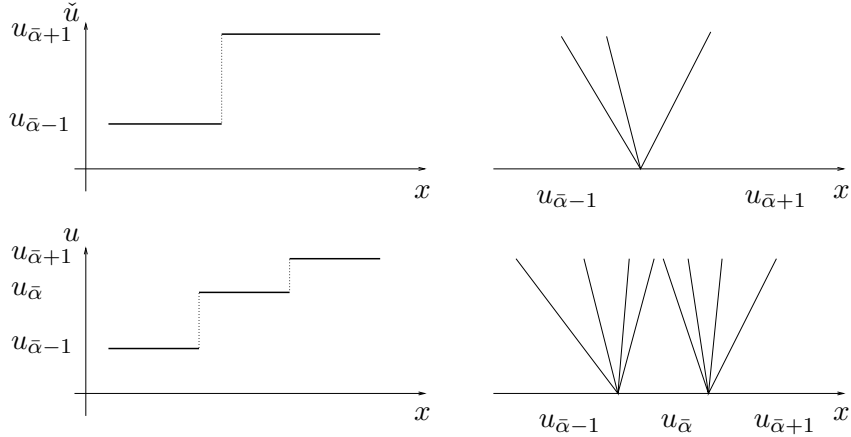}
    \end{psfrags}
    \caption{Proof of Proposition~\ref{prop:reduce}: $u_{\bar\alpha}$
      is attained by $u$ and not by $\check u$.}
    \label{fig:Glimm}
  \end{figure}
  Then, the above inequalities follow from the usual Glimm interaction
  estimates~\cite{Glimm}, see Figure~\ref{fig:Glimm}.

  The general case follows recursively.
\end{proof}

The next lemma is a particular case
of~\cite[Theorem~10.1]{BressanLectureNotes}. However, the present
construction allows to consider only the case of piecewise constant
functions, allowing a much simpler proof.

\begin{lemma}
  \label{lem:Qlsc}
  The functionals $\mathbf{Q}$ and $\mathbf{\Upsilon}$, defined on
  $\mathcal{D}_\delta^*$, are lower semicontinuous with respect to the
  $\L1$ norm.
\end{lemma}

\begin{proof}
  We consider only $\mathbf{\Upsilon}$, the case of $\mathbf{Q}$ being
  similar.

  Let $u_\nu$ be a sequence in $\mathcal{D}_\delta^*$ converging in
  $\L1$ to $u=\sum_\alpha
  u_\alpha\chi_{[x_\alpha,x_{\alpha+1}[}\in\mathcal{D}_\delta^*$ as
  $\nu \to +\infty$. By possibly passing to a subsequence, we may
  assume that $\mathbf{\Upsilon} (u_\nu)$ converges to
  $\liminf_{\nu\to+\infty} \mathbf{\Upsilon} (u_\nu)$ and that $u_\nu$
  converges a.e.~to $u$.  Therefore, for all $\alpha=1, \ldots, N$, we
  can select points $y_\alpha \in \left]x_{\alpha},
    x_{\alpha+1}\right[$ so that $\lim_{\nu\to+\infty} u_\nu(y_\alpha)
  = u(y_\alpha)=u_\alpha$. Define
  \begin{displaymath}
    \check u_\nu 
    =
    \sum_{\alpha} u_\nu(y_\alpha) \, \chi_{\left[x_{\alpha},
        x_{\alpha+1}\right[} \,.
  \end{displaymath}
  By Proposition~\ref{prop:reduce}, $\mathbf{\Upsilon}(\check u_\nu)
  \leq \mathbf{\Upsilon}(u_\nu)$. The convergence $u_\nu (y_\alpha)
  \to u_\alpha$ for all $\alpha$ and Remark~\ref{rem:Lipschitz} allow
  to complete the proof.
\end{proof}

\begin{proposition}
  \label{prop:coincide}
  Let $u \in \mathcal{D}_\delta^*$. Then $\bar{\mathbf{Q}} (u) =
  \mathbf{Q} (u)$ and $\bar{\mathbf{\Upsilon}} (u) = \mathbf{\Upsilon}
  (u)$.
\end{proposition}

\begin{proof}
  We consider only $\mathbf{\Upsilon}$, the case of $\mathbf{Q}$ being
  similar.

  Since $u\in\mathcal{D}_\delta^*$, we have that $u \in B_\eta(u)$ for
  all $\eta >0$ and $\bar{\mathbf{\Upsilon}} (u) \leq
  \mathbf{\Upsilon} (u)$.  To prove the other inequality, recall that
  by the definition~\Ref{eq:Q} of $\bar{\mathbf{\Upsilon}}$, there
  exists a sequence $v_\nu$ of piecewise constant functions in
  $\mathcal{D}_\delta^*$ such that $v_\nu \to u$ in $\L1$ and
  $\mathbf{\Upsilon}(v_\nu) \to \bar{\mathbf{\Upsilon}}(u)$ as $\nu
  \to +\infty$. By Lemma~\ref{lem:Qlsc}, $\mathbf{\Upsilon} (u) \leq
  \liminf_{\nu\to +\infty} \mathbf{\Upsilon}(v_\nu) \leq
  \bar{\mathbf{\Upsilon}} (u)$, completing the proof.
\end{proof}

Therefore, in the sequel we write $\mathbf{Q}$ for $\bar{\mathbf{Q}}$
and $\mathbf{\Upsilon}$ for $\bar{\mathbf{\Upsilon}}$.

\smallskip

\noindent Since we will need the explicit dependence of the Stability
Functional on the various quantity it is made of, we introduce the
following notations. If $\delta \in \left]0, \delta_o\right[$, for any
$\bar v\in\mathcal{D}_\delta^*$, denote by $\bar\sigma_{x,i}$ the size
of the $i$--wave in the solution of the Riemann Problem with data
$\bar v(x-)$ and $\bar v(x+)$. Then define
\begin{displaymath}
  A_j^- [\bar v](x)
  =
  \sum_{y\leq x} \modulo{\bar\sigma_{y,j}},
  \quad 
  A_j^+ [\bar v](x)
  =
  \sum_{y> x} \modulo{\bar\sigma_{y,j}},
  \quad
  \mbox{ for } j = 1,\ldots,n\,.
\end{displaymath}
If the $i$-th characteristic field is linearly degenerate, then
define $\mathbf{A}_i$ as
\begin{displaymath}
  \mathbf{A}_i\left[\bar v\right]\left(q,x\right)
  =
  \sum_{1\leq j<i}A_j^+\left[\bar v\right](x) + 
  \sum_{ i<j\leq n}A_j^-\left[\bar v\right](x).
\end{displaymath}
While if the $i$-th characteristic field is genuinely nonlinear
\begin{eqnarray*}
  \mathbf{A}_i [\bar v] (q,x)
  & = &
  \sum_{1\leq j<i} A_j^+ [\bar v] (x) 
  + 
  \sum_{ i<j\leq n}A_j^- [\bar v] (x)
  \\
  & &
  +
  A_i^+ [\bar v] (x) \cdot \caratt{\left[0, +\infty\right[} (q)
  +
  A_i^- [\bar v] (x) \cdot \caratt{\left]-\infty, 0\right[} (q)
\end{eqnarray*}
Now choose $v,\tilde v$ piecewise constant in $\mathcal{D}_\delta^*$
and define the weights
\begin{equation}
  \label{def:w}
  \begin{array}{rcl}
    \mathbf{W}_i [v,\tilde v](q,x)
    & = &
    1
    +
    \kappa_1 \mathbf{A}_i [v] (q,x)
    +
    \kappa_1 \mathbf{A}_i [\tilde v] (-q,x)
    \\[2pt]
    & &
    +
    \kappa_1 \kappa_2 \left(\mathbf{Q} (v) + \mathbf{Q} (\tilde v) \right)\,.
  \end{array}
\end{equation}
the constants $\kappa_1$ and $\kappa_2$ being defined
in~\cite[Chapter~8]{BressanLectureNotes}. We may now define a slightly
modified version of the stability functional,
see~\cite{BressanYangLiu, LiuYang1, LiuYang3} and
also~\cite[Section~8.1]{BressanLectureNotes}.  Namely, we give a
similar functional defined on all piecewise constant functions and
without any reference to both $\epsilon$--approximate front tracking
solutions and non physical waves.

Define implicitly the function $\mathbf{q}(x) \equiv \left( q_1(x),
  \ldots, q_n(x) \right)$ by
\begin{displaymath}
  \tilde v(x)
  =
  \mathbf{S} \left( \mathbf{q}(x) \right) \left(v(x)\right)
\end{displaymath}
with $\mathbf{S}$ as in~\Ref{eq:S}. The stability functional $\gf$
is
\begin{equation}
  \label{eq:Phi}
  \gf(v,\tilde v)
  =
  \sum_{i=1}^n \int_{-\infty}^{+\infty}
  \modulo{q_i(x)} \cdot 
  \mathbf{W}_i [v,\tilde v] \left(q_i(x),x\right) \, dx.
\end{equation}

We stress that $\gf$ is slightly different from the functional $\Phi$
defined in~\cite[formula~(8.6)]{BressanLectureNotes}. Indeed, here
\emph{all} jumps in $v$ or in $\tilde v$ are considered. There, on the
contrary, exploiting the structure of $\epsilon$-approximate front
tracking solutions, see~\cite[Definition~7.1]{BressanLectureNotes}, in
the definition of $\Phi$ the jumps due to non physical waves are
neglected when defining the weights $A_i$ and are considered as
belonging to a fictitious $(n+1)$-th family in the
definition~\cite[formula~(7.54)]{BressanLectureNotes} of $Q$. To
stress this difference, in the sequel we denote by $\Phi^\epsilon$ the
stability functional as presented
in~\cite[Chapter~8]{BressanLectureNotes}.

\begin{remark}
  \label{rem:ContPhi}
  For fixed $x_1< \ldots < x_{N+1}$, $\tilde x_1< \ldots < \tilde
  x_{\tilde N+1}$, the map
  \begin{displaymath}
    \left(
      \begin{array}{c}
        u_1, \ldots, u_N
        \\
        \tilde u_1, \ldots, \tilde u_{\tilde N}
      \end{array}
    \right) 
    \mapsto 
    \mathbf{\Phi} \left( 
      \sum_{\alpha=1}^N u_\alpha \, 
      \chi_{\left[x_{\alpha}, x_{\alpha+1}\right[},
      \sum_{\alpha=1}^{\tilde N} \tilde u_\alpha \, 
      \chi_{\left[\tilde x_{\alpha}, \tilde x_{\alpha+1}\right[}
    \right)
  \end{displaymath}
  is continuous. Indeed, since both maps $q \mapsto q \caratt{\left[0,
      +\infty\right[} (q)$ and $q \mapsto q \caratt{\left]-\infty,
      0\right[} (q)$ are Lipschitz, for any fixed $x \in \reali$ the
  integrand in~\Ref{eq:Phi} depends continuously on
  $\{u_\alpha\}_{\alpha=1}^N$, $\{\tilde u_\alpha\}_{\alpha=1}^{\tilde
    N}$ and the Dominated Convergence Theorem applies.
\end{remark}

We now move towards the extension of $\gf$ to $\mathcal{D}_{\delta}$.
Define
\begin{displaymath}
  \mathbf{\Xi}_\eta (u, \tilde u)
  =
  \inf \left\{
    \gf(v,\tilde v) \colon 
    v \in B_\eta (u) \mbox{ and }
    \tilde v \in B_\eta (\tilde u)
  \right\}
\end{displaymath}
The map $\eta \to \mathbf{\Xi}_\eta(u,\tilde u)$ is non increasing.
Thus, we may finally define
\begin{equation}
  \label{eq:Xi}
  \mathbf{\Xi}(u,\tilde u)
  =
  \lim_{\eta \to 0+} \mathbf{\Xi}_\eta (u, \tilde u) 
  =
  \sup_{\eta > 0} \mathbf{\Xi}_\eta (u, \tilde u)
  \,.
\end{equation}

We are now ready to state the main result of this section.

\begin{theorem}
  \label{thm:main}
  Let $f$ satisfy~\textbf{(F)}. The functional $\mathbf{\Xi} \colon
  \mathcal{D}_{\delta} \times \mathcal{D}_{\delta} \mapsto \left[0,
    +\infty\right[$ defined in~\Ref{eq:Xi} enjoys the
  properties~\textit{(i)}, \textit{(ii)} and~\textit{(iii)} in
  Theorem~\ref{thm:one}.
\end{theorem}

\noindent Here and in what follows, we denote by $C$ positive
constants dependent only on $f$ and $\delta_0$. We split the proof of
the above theorem in several steps.

\begin{lemma}
  \label{lem:inequality}
  For all $u,\tilde u \in \mathcal{D}_\delta^*$, one has $\mathbf{\Xi}
  (u,\tilde u) \leq \gf (u,\tilde u)$.
\end{lemma}

\begin{proof}
  By the definition~\Ref{eq:Beta} we have $u\in B_\eta(u)$ and $\tilde
  u \in B_\eta (\tilde u)$ for all $\eta >0$, hence $\mathbf{\Xi}_\eta
  (u,\tilde u) \leq \gf(u,\tilde u)$ for all positive $\eta$. The
  lemma is proved passing to the limit $\eta \to 0+$.
\end{proof}

\begin{lemma}
  \label{lemma:weights}
  Let $u,\;\check u\in\mathcal{D}_\delta^*$ and $q\in\reali$. Assume
  that $\check u$ is given by
  \begin{displaymath}
    \check u
    =
    \sum_{\alpha=1}^N u(y_\alpha) \chi_{\left[x_\alpha,x_{\alpha+1}\right[}
  \end{displaymath}
  where $x_1<\ldots<x_{N+1}$ and $y_\alpha\in [x_\alpha,x_{\alpha+1}[$
  are given real points.  Then,
  \begin{displaymath}
    \mathbf{A}_i\left[\check u\right](q,x)
    +
    \kappa_2 \mathbf{Q}(\check u) 
    \leq
    \mathbf{A}_i\left[ u\right](q,x)
    +
    \kappa_2 \mathbf{Q}(u) + C \cdot \norma{\check  u(x)-u(x)} \, .
  \end{displaymath}
\end{lemma}

\begin{proof}
  Fix $\bar x \in \reali$ and prove the above inequality passing from
  $u$ to $\check u$ recursively applying three elementary operations:

  \smallskip
  \noindent\textbf{1.} $w'$ is obtained from $w$ only shifting the
  position of the points of jump but without letting any point of jump
  cross $\bar x$.  More formally, if $w = \sum_\alpha w_\alpha
  \chi_{\left[\xi_\alpha, \xi_{\alpha+1}\right[}$ with $\xi_\alpha <
  \xi_{\alpha+1}$, $w' = \sum_\alpha w_\alpha
  \chi_{\left[\xi_\alpha^\prime, \xi_{\alpha+1}^\prime\right[}$ with
  $\xi_\alpha^\prime < \xi_{\alpha+1}^\prime$ and moreover $\bar x \in
  \left[\xi_{\bar\alpha}, \xi_{\bar\alpha+1}\right[ \cap
  \left[\xi'_{\bar\alpha}, \xi'_{\bar\alpha+1} \right[$ for a suitable
  $\bar\alpha$, then
  \begin{displaymath}
    \mathbf{A}_i\left[w'\right](q,\bar x)+\kappa_2 
    \mathbf{Q}(w')= \mathbf{A}_i\left[ 
      w\right](q,\bar x)+\kappa_2 \mathbf{Q}(w)
  \end{displaymath}
  Indeed, if all the jumps stay unchanged and no shock crosses $\bar
  x$, then nothing changes in the definition of $\mathbf{A}_i$ and
  $\mathbf{Q}$.

  \smallskip
  \noindent\textbf{2.} $w'$ is obtained from $w$ removing a value
  attained by $w$ on an interval not containing $\bar x$, see
  Figure~\ref{fig:Glimm2}.  More formally, if $w = \sum_\alpha
  w_\alpha \chi_{\left[\xi_\alpha, \xi_{\alpha+1}\right[}$ with
  $\xi_\alpha < \xi_{\alpha+1}$ and $\bar x \not\in \;
  [\xi_{\bar\alpha}, \xi_{\bar\alpha+1}[$, then
  \begin{displaymath}
    w' = \sum_{\alpha\not=\bar\alpha} 
    w_\alpha \, \chi_{\left[\xi_\alpha, \xi_{\alpha+1}\right[} 
    + 
    w_{\bar\alpha-1} \, 
    \chi_{\left[\xi_{\bar\alpha}, \xi_{\bar\alpha+1}\right[}
  \end{displaymath}
  or
  \begin{displaymath}
    w' = \sum_{\alpha\not=\bar\alpha} 
    w_\alpha \, \chi_{\left[\xi_\alpha, \xi_{\alpha+1}\right[} 
    + 
    w_{\bar\alpha+1} \, 
    \chi_{\left[\xi_{\bar\alpha}, 
        \xi_{\bar\alpha+1}\right[} \,.
  \end{displaymath}
  In both cases,
  \begin{displaymath}
    \mathbf{A}_i\left[ w'\right](q,\bar x)+
    \kappa_2 \mathbf{Q}(w') 
    \leq
    \mathbf{A}_i[w](q,\bar x)+
    \kappa_2 \mathbf{Q}(w)  \,.
  \end{displaymath}
  Indeed, consider for example the situation in
  Figure~\ref{fig:Glimm2}.
  \begin{figure}
    \centering
    \begin{psfrags}
      \psfrag{t}{$t$} \psfrag{m}{$w_{\bar\alpha-1}$}
      \psfrag{b}{$w_{\bar\alpha}$} \psfrag{xa}{$\xi_{\bar\alpha}$}
      \psfrag{xb}{$\xi_{\bar\alpha+1}$} \psfrag{xh}{$\bar x$}
      \psfrag{p}{$w_{\bar\alpha+1}$} \psfrag{x}{$x$} \psfrag{u}{$w$}
      \psfrag{uc}{$w'$}
      \includegraphics[width=11cm]{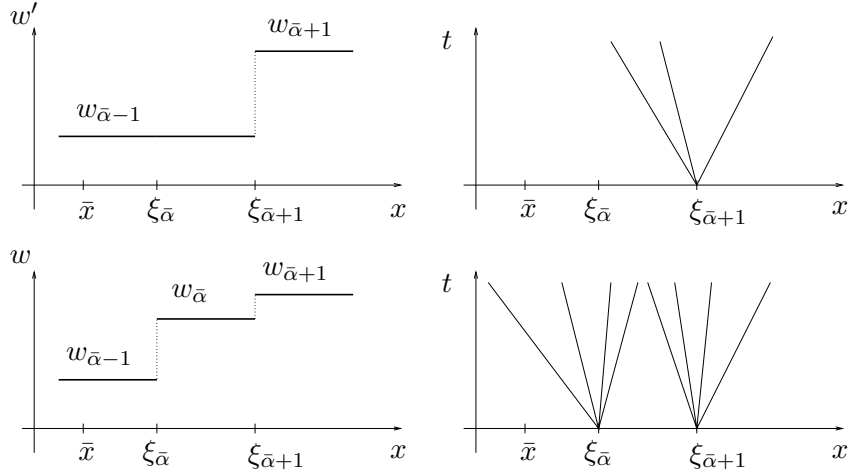}
    \end{psfrags}
    \caption{Exemplification of point 2 in the proof of
      Lemma~\ref{lemma:weights}.}
    \label{fig:Glimm2}
  \end{figure}
  The two jumps at the points $\xi_{\bar \alpha}$ and $\xi_{\bar\alpha
    +1}$ in $w$ are substituted by a single jump in $w'$ at the point
  $\xi_{\bar\alpha+1}$. The points $\xi_{\bar \alpha}$ and
  $\xi_{\bar\alpha +1}$ are both to the right of $\bar x$, therefore
  the waves in $w'$ at the point $\xi_{\bar\alpha+1}$ which appear in
  $\mathbf{A}_i[w'](q,\bar x)$ are of the same families of the waves
  in $w$ at the points $\xi_{\bar \alpha}$ and $\xi_{\bar\alpha +1}$
  which appear in $\mathbf{A}_i[w](q,\bar x)$. Since all the other
  waves in $\mathbf{A}_i$ are left unchanged we have
  \begin{displaymath}
    \mathbf{A}_i [w'] (q,\bar x) - \mathbf{A}_i [w] (q,\bar x) 
    \leq
    \sum_{j=1}^n 
    \modulo{\sigma_{\xi_{\bar\alpha+1},j}^\prime -
      \sigma_{\xi_{\bar\alpha},j} - \sigma_{\xi_{\bar\alpha+1},j}}
  \end{displaymath}
  where
  \begin{displaymath}
    \begin{array}{lclcl}
      \sigma_{\xi_{\bar\alpha+1},j}^\prime
      & = &
      E_j \left( w' \left( \xi_{\bar\alpha+1}- \right),
        w'\left(\xi_{\bar\alpha+1}+ \right) \right) 
      & = &
      E_j \left( w_{\bar\alpha-1},w_{\bar\alpha+1}\right)
      \\
      \sigma_{\xi_{\bar\alpha+1},j}
      & = &
      E_j \left( w\left(\xi_{\bar\alpha+1}- \right),
        w\left(\xi_{\bar\alpha+1}+ \right) \right)
      & = &
      E_j \left( w_{\bar\alpha},w_{\bar\alpha+1}\right)
      \\
      \sigma_{\xi_{\bar\alpha},j}
      & = &
      E_j \left( w\left(\xi_{\bar\alpha}-\right),
        w\left(\xi_{\bar\alpha}+ \right) \right)
      & = &
      E_j \left( w_{\bar\alpha-1},w_{\bar\alpha}\right) \,.
    \end{array}
  \end{displaymath}
  Therefore, the increase in $\mathbf{A}_i$ evaluated at $\bar x$ is
  bounded by the interaction potential between the waves at
  $\xi_{\bar\alpha}$ and those at $\xi_{\bar\alpha+1}$ and is
  compensated by the decrease in $\kappa_2 \mathbf{Q}$, as in the
  standard Glimm interaction estimates.

  \smallskip
  \noindent\textbf{3.} $w'$ is obtained from $w$ changing the value
  assumed by $w$ in the interval containing $\bar x$.  More formally,
  if $w = \sum_\alpha w_\alpha \chi_{\left[\xi_\alpha,
      \xi_{\alpha+1}\right[}$ with $\xi_\alpha < \xi_{\alpha+1}$ and
  $\bar x \in \left[\xi_{\bar\alpha}, \xi_{\bar\alpha+1}\right[$, then
  \begin{displaymath}
    w' = \sum_{\alpha\not=\bar\alpha} w_\alpha \chi_{\left[\xi_\alpha,
        \xi_{\alpha+1}\right[} + w_{\bar \alpha}^\prime
    \chi_{\left[\xi_{\bar\alpha}, \xi_{\bar\alpha+1}\right[} \,.
  \end{displaymath}
  In this case
  \begin{eqnarray*}
    \mathbf{A}_i\left[ w'\right](q,\bar x)+
    \kappa_2 \mathbf{Q}(w') 
    & \leq &
    \mathbf{A}_i[w](q,\bar x)+
    \kappa_2 \mathbf{Q}(w) 
    + C \cdot \norma{w_{\bar\alpha}-w_{\bar\alpha}^\prime}
    \\
    & \leq &
    \mathbf{A}_i[w](q,\bar x)+
    \kappa_2 \mathbf{Q}(w) 
    + C \cdot \norma{w(\bar x)-w^\prime(\bar x)} \,.
  \end{eqnarray*}
  This inequality follows from the Lipschitz dependence of
  $\mathbf{A}_i[w](q,\bar x)(\bar x)$ and $\mathbf{Q}(w)$ on
  $w_{\bar\alpha}$ with a Lipschitz constant independent from the
  number of jumps, see Remark~\ref{rem:Lipschitz}.

  \smallskip

  Now for $\bar x\in [x_{\bar \alpha},x_{\bar \alpha+1}[$ we can pass
  from $u$ to the function $\bar w$ defined by
  \begin{displaymath}
    \bar w 
    = 
    \sum_{\alpha\not=\bar\alpha} 
    u(y_\alpha) \chi_{\left[x_\alpha, x_{\alpha+1}\right[} 
    + 
    u(\bar x) \chi_{\left[x_{\bar\alpha}, x_{\bar\alpha+1}\right[}
  \end{displaymath}
  applying the first two steps a certain number of times. We obtain
  \begin{displaymath}
    \mathbf{A}_i [\bar w] (q,\bar x) + \kappa_2 \mathbf{Q}(\bar w) 
    \leq
    \mathbf{A}_i[u](q,\bar x) + \kappa_2 \mathbf{Q}(u)  \,.
  \end{displaymath}
  Finally with the third step we go from $\bar w$ to $\check u$
  obtaining the estimate:
  \begin{eqnarray*}
    \mathbf{A}_i [\check u] (q,\bar x) + \kappa_2 \mathbf{Q}(\check u)
    & \leq &
    \mathbf{A}_i [\bar w] (q,\bar x) + \kappa_2 \mathbf{Q}(\bar w) 
    + C \cdot \norma{\check u(\bar x)-\bar w(\bar x)}
    \\
    & \leq &
    \mathbf{A}_i [u] (q,\bar x) + \kappa_2 \mathbf{Q} (u) 
    + C \cdot \norma{\check u(\bar x)-u(\bar x)}
  \end{eqnarray*}
  which proves the lemma.
\end{proof}

  \begin{lemma}
    \label{lemma:reduce}
    Let $u,\;\check u,\;\tilde u,\; \check{\tilde
      u}\in\mathcal{D}_\delta^*$.  Assume that $\check u$ and
    $\check{\tilde u}$ are given by
    \begin{displaymath}
      \check 
      u=\sum_{\alpha=1}^Nu(y_\alpha)\chi_{[x_\alpha,x_{\alpha+1}[}
      ,\qquad
      \check{\tilde u}=
      \sum_{\alpha=1}^{\tilde 
        N}\tilde u(\tilde y_\alpha)\chi_{[\tilde x_\alpha,\tilde x_{\alpha+1}[}
    \end{displaymath}
    where $x_1<\ldots<x_{N+1}$, $y_\alpha\in [x_\alpha,x_{\alpha+1}[$,
    $\tilde x_1<\ldots<\tilde x_{\tilde N+1}$, $\tilde y_\alpha\in
    [\tilde x_\alpha,\tilde x_{\alpha+1}[$ are fixed real points.
    Then,
    \begin{displaymath}
      \mathbf{\Phi} \left(\check u,\check{\tilde u}\right)
      \leq
      \mathbf{\Phi} \left(u,\tilde u\right) 
      + 
      C\left(
        \norma{\check u-u}_{\L1} + \norma{\check{\tilde u}-\tilde
          u}_{\L1}
      \right) \,.
    \end{displaymath}
  \end{lemma}

\begin{proof}
  Introduce $\mathbf{q}(x) = \left(q_1(x), \ldots, q_n(x)\right)$ and
  $\mathbf{\check q}(x) = \left(\check q_1(x), \ldots, \check
    q_n(x)\right)$ by $\tilde u(x) = \mathbf{S}
  \left(\mathbf{q}(x)\right) \left(u(x)\right)$ and $\check{\tilde
    u}(x) = \mathbf{S} \left(\mathbf{\check q}(x)\right) \left(\check
    u(x)\right)$ with $\mathbf{S}$ defined in~\Ref{eq:S}.
  \begin{eqnarray*}
    & &
    \mathbf{\Phi} \left(\check u,\check{\tilde u}\right)
    -
    \mathbf{\Phi}\left(u,\tilde u\right)
    \\
    &=&
    \!\! \sum_{i=1}^n \!
    \int_{-\infty}^{+\infty} \!
    \left\{ 
      \modulo{\check q_i(x)}
      \mathbf{W}_i [\check u, \check{\tilde u}] 
      \left(\check q_i(x),x\right)
      -
      \modulo{q_i(x)}
      \mathbf{W}_i [\check u, \check{\tilde u}]
      \left( q_i(x),x\right)
    \right\}
    dx
    \\
    & &
    +
    \sum_{i=1}^n
    \int_{-\infty}^{+\infty} \modulo{q_i(x)}
    \left\{
      \mathbf{W}_i [\check u, \check{\tilde u}]
      \left( q_i(x),x\right)
      -
      \mathbf{W}_i [u, {\tilde u}] \left( q_i(x),x\right)
    \right\}
    \, dx \,.
  \end{eqnarray*}
  Since the map $q \mapsto q \cdot \mathbf{W}_i [\check u,
  \check{\tilde u}] (q,x)$ is uniformly Lipschitz, the first
  integral is bounded by
  \begin{displaymath}
    C \sum_{i=1}^n
    \int_{-\infty}^{+\infty} \modulo{\check q_i(x)-q_i(x)} \; dx
    \leq C 
    \left(
      \norma{\check u-u}_{\L1} + \norma{\check{\tilde u}-\tilde
        u}_{\L1}
    \right).
  \end{displaymath}
  Concerning the second integral, observe that by
  Lemma~\ref{lemma:weights}
  \begin{eqnarray*}
    & &
    \mathbf{W}_i [\check u, \check{\tilde u}]
    \left( q_i(x),x\right)
    -
    \mathbf{W}_i [u, {\tilde u}]
    \left( q_i(x),x\right)
    \\
    & &\qquad\qquad
    \leq
    \kappa_1 C 
    \left\{
      \norma{\check u(x)-u(x)}
      +
      \norma{\check{\tilde u}(x)-\tilde u(x)}
    \right\}
  \end{eqnarray*}
  and, since $\modulo{q_i(x)}$ is uniformly bounded, the Lemma is
  proved.
\end{proof}

\begin{proposition}
  \label{lemma5.1}
  The functional $\gf$, defined on $\mathcal{D}_\delta^*\times\mathcal{D}_\delta^*$, is lower
  semicontinuous with respect to the $\L1$ norm.
\end{proposition}

\begin{proof}
  Fix $u,\tilde u$ in $\mathcal{D}_\delta^*$. Choose two sequences of
  piecewise constant maps $u_\nu,\tilde u_\nu$ in
  $\mathcal{D}_\delta^*$ converging to $u,\tilde u$ in $\L1$. We want
  to show that $\mathbf{\Phi}(u,\tilde u) \leq \liminf_{\nu\to
    +\infty} \mathbf{\Phi} (u_\nu, \tilde u_\nu)$. Call
  $l=\liminf_{\nu\to +\infty} \mathbf{\Phi} (u_\nu, \tilde u_\nu)$ and
  note that, up to subsequences, we may assume that $\lim_{\nu\to
    +\infty} \mathbf{\Phi} (u_\nu, \tilde u_\nu) = l$. By possibly
  selecting a further subsequence, we may also assume that both
  $u_\nu$ and $\tilde u_\nu$ converge a.e.~to $u$ and $\tilde u$.

  Let $\{x_1,\ldots,x_{N+1}\}$ be the set of the jump points in $u$
  and $\tilde u$ and write
  \begin{displaymath}
    u = \sum_{\alpha=1}^{N} u_\alpha \, \chi_{\left[x_\alpha,
        x_{\alpha+1}\right[},\qquad \tilde u = \sum_{\alpha=1}^{N}
    \tilde u_\alpha \, \chi_{\left[x_\alpha, x_{\alpha+1}\right[} \,.
  \end{displaymath}
  For all $\alpha$, select $y_\alpha \in \left]x_{\alpha},
    x_{\alpha+1}\right[$ so that as $\nu \to +\infty$, the sequence
  $u_\nu(y_\alpha)$ converges to $u(y_\alpha)=u_\alpha$ and $\tilde
  u_\nu(y_\alpha)$ converges to $\tilde u(y_\alpha) = \tilde
  u_\alpha$.  Introduce the piecewise constant function $\check u_\nu
  = \sum_\alpha u_\nu (y_\alpha)\, \chi_{\left[x_\alpha,
      x_{\alpha+1}\right[}$. Let $\check{\tilde u}_\nu$ be defined
  analogously. Observe that $\check u_\nu$ and $\check{\tilde u}_\nu$
  converge pointwise and in $\L1$ to respectively $u$ and $\tilde u$.
  Remark~\ref{rem:ContPhi} implies $\lim_{\nu\rightarrow
    +\infty}\mathbf{\Phi}\left(\check u_\nu,\check{\tilde
      u}_\nu\right)=\mathbf{\Phi}\left( u,\tilde u\right)$, while
  Lemma~\ref{lemma:reduce} implies
  \begin{displaymath}
    \mathbf{\Phi} (\check u_\nu,\check{\tilde u}_\nu)
    \leq 
    \mathbf{\Phi}\left( u_\nu,\tilde u_\nu\right)
    + C \left(\norma{\check u_\nu-u_\nu}_{\L1}
      + \norma{\check{\tilde u}_\nu - \tilde u_\nu}_{\L1}\right) \,.
  \end{displaymath}
  Therefore, passing to the limit $\nu \to +\infty$, the proof is
  completed:
  \begin{eqnarray*}
    \mathbf{\Phi} (u,\tilde u)
    &=&
    \mathbf{\Phi} ( u,\tilde u)
    -
    \mathbf{\Phi} (\check u_\nu,\check{\tilde u}_\nu)
    +
    \mathbf{\Phi} (\check u_\nu,\check{\tilde u}_\nu)
    -
    \mathbf{\Phi} (u_\nu,\tilde u_\nu)
    +
    \mathbf{\Phi} (u_\nu,\tilde u_\nu)
    \\
    & \leq &
    \mathbf{\Phi} (u,\tilde u)
    -
    \mathbf{\Phi} (\check u_\nu,\check{\tilde u}_\nu)
    +
    C\left(
      \norma{\check u_\nu-u_\nu}_{\L1}
      +
      \norma{\check{\tilde u}_\nu-\tilde u_\nu}_{\L1}
    \right)
    \\
    & &
    +\mathbf{\Phi} (u_\nu,\tilde u_\nu) \,.
    \vspace{-2\baselineskip}
  \end{eqnarray*}
  \vspace{-\baselineskip}
\end{proof}

\begin{lemma}
  \label{lem:prima}
  For all $u,\tilde u \in \mathcal{D}_\delta^*$, one has
  $\mathbf{\Xi}(u,\tilde u) \geq \gf(u,\tilde u)$.
\end{lemma}

\begin{proof}
  By the definition~\Ref{eq:Xi} of $\mathbf{\Xi}$, for all $u, \tilde
  u \in \mathcal{D}_\delta^*$, there exist sequences $v_\nu, \tilde
  v_\nu$ of piecewise constant functions such that for $\nu \to
  +\infty$ we have $v_\nu \to u$, $\tilde v_\nu \to u$ in $\L1$ and
  $\gf(v_\nu,\tilde v_\nu) \to \mathbf{\Xi}(u,\tilde u)$. Hence, by
  Proposition~\ref{lemma5.1}
  \begin{displaymath}
    \gf(u, \tilde u) 
    \leq 
    \liminf_{\nu\to+\infty} \gf(v_\nu, \tilde v_\nu) 
    =
    \mathbf{\Xi} (u, \tilde u) \,,
  \end{displaymath}
  completing the proof.
\end{proof}

Lemma~\ref{lem:inequality} and Lemma~\ref{lem:prima} together yield
the following proposition.

\begin{proposition}
  \label{prop:Equal}
  For all $u$, $\tilde u$ in $\mathcal{D}_\delta^*$, one has
  $\mathbf{\Xi}(u,\tilde u) = \gf(u, \tilde u)$.
\end{proposition}

Thanks to the definition~\Ref{eq:Xi} of $\mathbf{\Xi}$, we obtain

\begin{proposition}
  \label{prop:lsc}
  The functional $\mathbf{\Xi} \colon
  \mathcal{D}_\delta\times\mathcal{D}_\delta\mapsto \reali$ is lower
  semicontinuous with respect to the $\L1$ norm.
\end{proposition}

\begin{proof}
  Fix $u$ and $\tilde u$ in $\mathcal{D}_\delta$.  Let $u_\nu$,
  respectively $\tilde u_\nu$, be a sequence in $\mathcal{D}_\delta$
  converging to $u$, respectively $\tilde u$.  Define $\epsilon_\nu =
  \norma{u_\nu - u}_{\L1} + \norma{\tilde u_\nu - \tilde u}_{\L1} +
  1/\nu$. Then, for each $\nu$, there exist piecewise constant $v_\nu
  \in B_{\epsilon_\nu} (u_\nu)$, respectively $\tilde v_\nu \in
  B_{\epsilon_\nu} (\tilde u_\nu)$, such that
  \begin{equation}
    \label{eq:inequality}
    \gf( v_\nu, \tilde v_\nu) 
    \leq 
    \mathbf{\Xi}_{\epsilon_\nu} (u_\nu, \tilde u_\nu) + \epsilon_\nu
    \leq \mathbf{\Xi}(u_\nu, \tilde u_\nu)+\epsilon_\nu \,.
  \end{equation}
  Moreover
  \begin{displaymath}
    \begin{array}{rcccl}
      \norma{v_\nu - u}_{\L1}
      & \leq &
      \norma{v_\nu - u_\nu}_{\L1} +
      \norma{u_\nu - u}_{\L1}
      & < &
      2 \epsilon_\nu
      \\
      \norma{\tilde v_\nu - \tilde u}_{\L1}
      & \leq &
      \norma{\tilde v_\nu - \tilde u_\nu}_{\L1} +
      \norma{\tilde u_\nu - \tilde u}_{\L1}
      & < &
      2 \epsilon_\nu
    \end{array}
  \end{displaymath}
  so that $v_\nu \in B_{2\epsilon_\nu} (u)$ and $\tilde v_\nu \in
  B_{2\epsilon_\nu} (\tilde u)$. Hence, $\mathbf{\Xi}_{2\epsilon_\nu}
  (u, \tilde u) \leq \gf(v_\nu, \tilde v_\nu)$.
  Using~\Ref{eq:inequality}, we obtain $\mathbf{\Xi}_{2\epsilon_\nu}
  (u, \tilde u) \leq \mathbf{\Xi}(u_\nu, \tilde u_\nu)+\epsilon_\nu$.
  Finally, passing to the lower limit for $\nu \to +\infty$, we have
  $\mathbf{\Xi}(u,\tilde u) \leq \liminf_{\nu\to+\infty} \mathbf{\Xi}
  (u_\nu,\tilde u_\nu)$.
\end{proof}

In the next proposition, we compare the functional $\gf$ defined
in~\Ref{eq:Phi} with the stability functional $\Phi^\epsilon$ as
defined in~\cite[formula~(8.6)]{BressanLectureNotes}

\begin{proposition}
  \label{lem:diet}
  Let $\delta>0$. Then, there exists a positive $C$ such that for all
  $\epsilon > 0$ sufficiently small and for all $\epsilon$-approximate
  front tracking solutions $w(t,x), \tilde w (t,x)$ of~\Ref{eq:HCL}
  \begin{displaymath}
    \modulo{
      \gf  \left( w(t, \cdot), \tilde w(t, \cdot) \right) - 
      \Phi^\epsilon ( w, \tilde w) (t)
    }
    \leq
    C \cdot \epsilon \cdot \norma{w(t,\cdot) - \tilde w(t,\cdot)}_{\L1}\,.
  \end{displaymath}
\end{proposition}

\begin{proof}
  Setting $\tilde w(t,x)=\mathbf{S}\left(\mathbf{q}(t,x)\right) \left(
    w(t,x)\right)$ and omitting the explicit time dependence in the
  integrand, we have:
  \begin{eqnarray*}
    & &
    \modulo{
      \gf  \left( w(t, \cdot), \tilde w(t, \cdot) \right) - 
      \Phi^\epsilon ( w, \tilde w) (t)
    }
    \\
    & & 
    \qquad \qquad \leq
    \int_{\reali} \sum_{i=1}^n \modulo{q_i(x)} \,
    \modulo{\mathbf{W}_i [w,\tilde w] \left(q_i(x),x\right) - W_i(x)} \, dx \,.
  \end{eqnarray*}
  We are thus lead to estimate
  \begin{eqnarray*}
    & &
    \modulo{\mathbf{W}_i[w,\tilde w] \left(q_i(x),x\right) - W_i(x)}
    \\
    & &
    \qquad\qquad
    \leq
    \kappa_1 
    \modulo{
      \mathbf{A}_i[w]        \left( q_i(x),x\right) + 
      \mathbf{A}_i[\tilde w] \left(-q_i(x),x\right) - 
      A_i(x)}
    \\
    & & 
    \qquad\qquad\qquad
    +
    \kappa_1 \kappa_2 \modulo{\mathbf{Q}(w) - Q(w)} +
    \kappa_1 \kappa_2 \modulo{\mathbf{Q}(\tilde w) - Q(\tilde w)} \,.
  \end{eqnarray*}
  The second and third terms in the right hand side are each bounded
  as in~\cite[formula~(7.100)]{BressanLectureNotes} by $C \,
  \epsilon$.  Concerning the former one, recall that, except when $q_i(x)=0$ or 
  on a finite number of points where $w$ or $\tilde w$ have jumps, $A_i$ and
  $\mathbf{A}_i$ differ only in the absence of non physical waves in
  $A_i$. In other words, physical jumps are counted in the same way in
  both $A_i$ and $\mathbf{A}_i$ while non physical waves appear in
  $\mathbf{A}_i$ but not in $A_i$. Therefore, the former term is almost 
  everywhere bounded, when $q_i(x)\not=0$,
  by the sum of the strengths of all non physical waves,
  i.e.~$C\, \epsilon$ by~\cite[formula~(7.11)]{BressanLectureNotes}.
  Finally, using~\cite[formula~(8.5)]{BressanLectureNotes}:
  \begin{equation}
    \label{eq:equivalence}
    \frac{1}{C} \cdot \norma{w(t,x) -\tilde w(t,x)}
    \leq
    \sum_{i=1}^n \modulo{q_i(t,x)}
    \leq
    C \cdot \norma{w(t,x) -\tilde w(t,x)}
  \end{equation}
  we complete the proof with the following estimate:
  \begin{displaymath}
    \modulo{
      \gf  \left( w(t, \cdot), \tilde w(t, \cdot) \right) - 
      \Phi^\epsilon ( w, \tilde w) (t)
    }
    \leq
    C\, \epsilon \int_{\reali} \sum_{i=1}^n \modulo{q_i(x)}\, dx
    \leq
    C\, \epsilon \norma{w-\tilde w}_{\L1}.
    \vspace{-\baselineskip}
  \end{displaymath}
\end{proof}

\begin{proofof}{Theorem~\ref{thm:main}}
  The estimates~\cite[formula~(8.5)]{BressanLectureNotes} 
  show that $\gf$ is equivalent to
  the $\L1$ distance between functions in $\mathcal{D}_\delta^*$.
  Indeed, if $\delta$ is sufficiently small, then $\mathbf{W}_i \in
  [1,2]$ for all $i = 1, \ldots, n$ and all $x \in \reali$, so that
  \begin{equation}
    \label{eq:Phi1}
    \frac{1}{C} \cdot \norma{v -\tilde v}_{\L1}
    \leq
    \gf(v,\tilde v)
    \leq
    2 C \cdot \norma{v -\tilde v}_{\L1} \,.
  \end{equation}
  To prove~\emph{(i)}, fix $u, \tilde u \in \mathcal{D}_\delta$ and
  choose $v \in B_\eta(u)$, $\tilde v \in B_\eta(\tilde u)$.
  By~\Ref{eq:Phi1},
  \begin{displaymath}
    \begin{array}{rcccl}
      \displaystyle
      \frac{1}{C} \cdot \left( \norma{u -\tilde u}_{\L1} -2\eta \right)
      & \leq &
      \gf(v,\tilde v)
      & \leq &
      \displaystyle
      2 C \cdot \left( \norma{u -\tilde u}_{\L1} +2\eta \right)
      \\[5pt]
      \displaystyle
      \frac{1}{C} \cdot \left( \norma{u -\tilde u}_{\L1} -2\eta \right)
      & \leq &
      \mathbf{\Xi}_\eta(u,\tilde u)
      & \leq &
      \displaystyle
      2 C \cdot \left( \norma{u -\tilde u}_{\L1} +2\eta \right) \,.
    \end{array}
  \end{displaymath}
  The proof of~\emph{(i)} is completed passing to the limit $\eta \to
  0+$.

  \smallskip

  To prove~\emph{(ii)}, fix $u,\tilde u \in \mathcal{D}_\delta$ and
  $\eta >0$.  Correspondingly, choose $v_\eta \in B_{\eta}(u)$ and
  $\tilde v_\eta \in B_{\eta}(\tilde u)$ satisfying
  \begin{equation}
    \label{eq:Eta}
    \mathbf{\Xi} (u, \tilde u) 
    \geq
    \mathbf{\Xi}_\eta (u, \tilde u) \geq \Phi( v_\eta, \tilde v_\eta) -\eta \,.
  \end{equation}

  Let now $\epsilon >0$ and introduce the $\epsilon$-approximate
  solutions $v^\epsilon_\eta$ and $\tilde v^\epsilon_\eta$ with
  initial data $v^\epsilon_\eta(0,\cdot) = v_\eta$ and $\tilde
  v^\epsilon_\eta(0,\cdot) = \tilde v_\eta$. Note that for $\epsilon$
  sufficiently small
  \begin{displaymath}
    \begin{array}{rcccccl}
      \displaystyle
      \mathbf{\Upsilon} \left( v^\epsilon_\eta(t)\right)
      & \leq &
      \displaystyle
      \Upsilon^\epsilon \left( v^\epsilon_\eta\right)(t) + C \epsilon
      & \leq &
      \Upsilon^\epsilon(v^\epsilon_\eta)(0) + C \epsilon
      \\
      & \leq &
      \displaystyle
      \mathbf{\Upsilon} (v_\eta) + C \epsilon
      & < &
      \displaystyle
      \delta
    \end{array}
  \end{displaymath}
  and an analogous inequality holds for $\tilde v^\epsilon_\eta$.
  Therefore $v_\eta^\epsilon(t),\,\tilde v_\eta^\epsilon(t) \in
  \mathcal{D}_\delta^*$. Here we denoted with $\Upsilon^\epsilon$ the
  sum $V+C_0 Q$ defined on $\epsilon$--approximate wave front tracking
  solutions (see~\cite[formul\ae~(7.53),
  (7.54)]{BressanLectureNotes}).  We may thus apply
  Proposition~\ref{prop:Equal}, Proposition~\ref{lem:diet} and the
  main result in~\cite[Chapter 8]{BressanLectureNotes}, that
  is~\cite[Theorem 8.2]{BressanLectureNotes}, to obtain
  \begin{eqnarray*}
    \mathbf{\Xi} 
    \bigl( v^\epsilon_\eta(t), \tilde v^\epsilon_\eta(t) \bigr)
    \!\!\!\!
    & = &
    \mathbf{\Phi} \left( v^\epsilon_\eta(t), \tilde v^\epsilon_\eta(t) \right)
    \\
    & \leq &
    \Phi^\epsilon \left( v^\epsilon_\eta, \tilde v^\epsilon_\eta \right)(t)
    + 
    C \epsilon \norma{v^\epsilon_\eta(t) - \tilde v^\epsilon_\eta(t)}_{\L1}
    \\
    & \leq &
    \Phi^\epsilon \left( v^\epsilon_\eta, \tilde v^\epsilon_\eta \right)(0)
    + 
    C \epsilon t
    +
    C \epsilon \norma{v^\epsilon_\eta(t) - \tilde v^\epsilon_\eta(t)}_{\L1}
    \\
    & \leq &
    \mathbf{\Phi} \left( v_\eta, \tilde v_\eta \right)
    \! + 
    C \epsilon t
    +
    C \epsilon \norma{v^\epsilon_\eta(t) - \tilde v^\epsilon_\eta(t)}_{\L1}
    \!\!\! +
    C \epsilon \norma{v_\eta - \tilde v_\eta}_{\L1}.
  \end{eqnarray*}
  Recall that as $\epsilon \to 0$
  by~\cite[Theorem~8.1]{BressanLectureNotes} $v^\epsilon_\eta(t) \to
  S_t v_\eta$ and $\tilde v^\epsilon_\eta(t) \to S_t \tilde v_\eta$.
  Hence, Proposition~\ref{prop:lsc} and~\Ref{eq:Eta} ensure that
  \begin{displaymath}
    \mathbf{\Xi} (S_t v_\eta,S_t \tilde v_\eta)
    \leq
    \liminf_{\epsilon \to 0+} 
    \mathbf{\Xi} \left( v^\epsilon_\eta(t), \tilde v^\epsilon_\eta (t) \right)
    \leq
    \mathbf{\Phi} \left( v_\eta, \tilde v_\eta \right)
    \leq
    \mathbf{\Xi}(u,\tilde u) +\eta \,.
  \end{displaymath}
  By the choice of $v_\eta$ and $\tilde v_\eta$, we have that $v_\eta
  \to u$ and $\tilde v_\eta \to \tilde u$ in $\L1$ as $\eta \to 0+$.
  Therefore, using the continuity of the SRS in $\L1$ and applying
  again Proposition~\ref{prop:lsc}, we may conclude that
  \begin{displaymath}
    \mathbf{\Xi} ( S_t u , S_t \tilde u) 
    \leq
    \liminf_{\eta \to 0+}     \mathbf{\Xi} (S_t v_\eta,S_t \tilde v_\eta)
    \leq
    \mathbf{\Xi}(u,\tilde u) \,,
  \end{displaymath}
  proving~\emph{(ii)}. The latter item~\emph{(iii)} follows from
  Proposition~\ref{prop:lsc}.
\end{proofof}

\Section{Wave Measures Formulation}
\label{sec:Wave}

Let $f$ satisfy~\textbf{(F)} and $u \in \mathcal{D}_\delta$ as defined
in~\Ref{def:2.6}. Since $\tv(u)$ is bounded, by possibly changing the
values of $u$ at countably many points, we can assume that $u$ is
right continuous. Its distributional derivative $\mu$ is then a vector
measure that can be decomposed into a continuous part $\mu_c$ and an
atomic one $\mu_a$.  For $i = 1 ,\ldots, n$, consider now the wave
measure
\begin{equation}
  \label{eq:mu}
  \mu_i (B) 
  = 
  \int_B l_i(u) \, d\mu_c 
  + 
  \sum_{x \in B} E_i \left( u(x-), u(x+) \right)
\end{equation}
where $l_i(u)$ is the left $i$-th eigenvector of $Df(u)$, $E_i$ is the
$i$-th component of the map $E$ defined at~\Ref{eq:E} and $B \subseteq
\reali$ is any Borel set. Here and in what follows, we assume that
$l_1, \ldots, l_n$ are normalized so that
\begin{displaymath}
  l_i(u) \cdot r_j(u) =
  \left\{
    \begin{array}{l@{\quad\mbox{ if }}rcl}
      1 & i & = & j
      \\
      0 & i & \neq & j
    \end{array}
  \right.
\end{displaymath}
with $r_j$ as in~\Ref{eq:parametrization}. Let $\mu_i^+$, respectively
$\mu_i^-$ be the positive, respectively negative, part of $\mu_i$ and
$\modulo{\mu_i} = \mu_i^+ + \mu_i^-$ be the total variation of
$\mu_i$. Aiming at a definition of the interaction potential by means
of the wave measures, introduce the measure
\begin{equation}\label{eq:defrho}
  \rho
  =
  \sum_{1\leq j < i \leq n} \modulo{\mu_i} \otimes \modulo{\mu_j}
  +
  \sum_{i=1}^n 
  \left(\mu_i^- \otimes \mu_i^-
    +
    \mu_i^+ \otimes \mu_i^-
    +
    \mu_i^- \otimes \mu_i^+
  \right)
\end{equation}
and, as in~\cite{BaitiBressan, BressanLectureNotes, BressanColombo2},
set
\begin{eqnarray}
  \label{eq:def2Q}
  \hat \mathbf{Q} (u)
  & = &
  \rho \left( \left\{ (x,y) \in \reali^2 \colon x < y \right\} \right)
  \\ \label{eq:def2U}
  \hat \mathbf{\Upsilon} (u)
  & = &
  \sum_{i=1}^n \modulo{\mu_i} (\reali)
  +
  C_0 \cdot \hat\mathbf{Q} (u) \,.
\end{eqnarray}
For $u, \tilde u$ in $\mathcal{D}_\delta$, we now define the
functional
\begin{equation}
  \label{eq:barPhi}
  \hat\mathbf{\Xi}(u,\tilde u)
  =
  \sum_{i=1}^n \int_{-\infty}^{+\infty}
  \modulo{q_i(x)} \, \hat\mathbf{W}_i(x) \, dx
\end{equation}
where the weights $\hat\mathbf{W}_i$ are defined by
\begin{equation}
  \hat\mathbf{W}_i(x)
  =
  1 + 
  \kappa_1 \hat\mathbf{A}_i(x) + 
  \kappa_1 \kappa_2 \left( \hat\mathbf{Q} (u) + \hat\mathbf{Q}(\tilde u) \right) \,.
  \label{def:barw}
\end{equation}
Here, $\kappa_1$ and $\kappa_2$ are as
in~\cite[Chapter~8]{BressanLectureNotes}, see also~\Ref{def:w}. By
means of the wave measures $\mu_i$ and $\tilde \mu_i$ of $u$ and
$\tilde u$, if the $i$-th field is linearly degenerate, define the
weights $\hat\mathbf{A}_i$ by
\begin{eqnarray*}
  \hat\mathbf{A}_i (x)
  & = &
  \sum_{1 \leq i < j \leq n} 
  \modulo{\mu_j} \left( \left]-\infty,x \right]\right)
  +
  \modulo{\tilde \mu_j} \left( \left]-\infty,x \right]\right)
  \\
  & &
  +
  \sum_{1 \leq j < i \leq n}
  \modulo{\mu_j} \left( \left]x, +\infty \right[\right)
  +
  \modulo{\tilde \mu_j} \left( \left]x, +\infty \right[\right)
\end{eqnarray*}
whereas in the genuinely nonlinear case we let
\begin{eqnarray*}
  \hat\mathbf{A}_i (x)
  & = &
  \sum_{1 \leq i < j \leq n} 
  \modulo{\mu_j} \left( \left]-\infty,x \right]\right)
  +
  \modulo{\tilde \mu_j} \left( \left]-\infty,x \right]\right)
  \\
  & &
  +
  \sum_{1 \leq j < i \leq n}
  \modulo{\mu_j} \left( \left]x, +\infty \right[\right)
  +
  \modulo{\tilde \mu_j} \left( \left]x, +\infty \right[\right)
  \\
  & &
  +
  \left\{
    \begin{array}{l@{\ \mbox{ if }}rcl}
      \displaystyle
      \sum_{i=1}^n
      \left(
        \modulo{\mu_i} \left( \left]-\infty, x \right]\right)
        +
        \modulo{\tilde \mu_i} \left( \left]x, +\infty \right[\right)
      \right)
      &
      q_i(x) & < & 0
      \\
      \displaystyle
      \sum_{i=1}^n
      \left(
        \modulo{\mu_i} \left( \left]x, +\infty \right[\right)
        +
        \modulo{\tilde \mu_i} \left( \left]-\infty, x \right]\right)
      \right)
      &
      q_i(x) & \geq & 0 \,.
    \end{array}
  \right.
\end{eqnarray*}

\noindent On $\mathcal{D}_\delta^*$, $\hat\mathbf{Q}$ and
$\hat\mathbf{\Upsilon}$ obviously coincide respectively with
$\mathbf{Q}$ and $\mathbf{\Upsilon}$, therefore also $\mathbf{\Xi}$,
$\mathbf{\Phi}$ and $\hat\mathbf{\Xi}$ all coincide on
$\mathcal{D}_\delta^*$. Below, we prove that $\hat\mathbf{\Upsilon} =
\mathbf{\Upsilon}$, $\hat\mathbf{Q} = \mathbf{Q}$ and $\mathbf{\Xi} =
\hat\mathbf{\Xi}$ on all $\mathcal{D}_{\delta}$

The following result is a strengthened version
of~\cite[Lemma~10.1]{BressanLectureNotes}.

\begin{lemma}
  \label{lem:diam}
  There exists a positive $C$ such that for all $u \in
  \mathcal{D}_\delta$, $i = 1, \ldots, n$ and $a,b \in \reali$ with $a
  < b$
  \begin{displaymath}
    \modulo{
      E_i \left( u(a+), u(b-) \right) 
      - 
      \mu_i \left( \left]a, b \right[ \right)
    }
    \leq
    C 
    \cdot
    \diam \left( u \left( \left]a,b\right[ \right) \right)
    \cdot
    \modulo{\mu} \left( \left]a,b\right[ \right)
  \end{displaymath}
\end{lemma}

\begin{proof}
  We use below the following estimate,
  see~\cite[p.~201]{BressanLectureNotes}, valid for all states $u,
  \tilde u$,
  \begin{equation}
    \label{eq:Taylor}
    \modulo{E_i(u,\tilde u) - l_i(u) \cdot (\tilde u - u)}
    \leq
    C \, \modulo{u-\tilde u}^2 \,.
  \end{equation}
  By the triangle inequality,
  \begin{eqnarray*}
    & &
    \modulo{
      E_i \left( u(a+), u(b-) \right) 
      - 
      \mu_i \left( \left]a, b \right[ \right)
    }
    \\
    & &
    \qquad\qquad
    \leq
    \modulo{
      E_i \left( u(a+), u(b-) \right) 
      - 
      l_i(a+) \cdot \left( u(b-) - u(a+) \right)
    }
    \\
    & &
    \qquad\qquad\qquad
    +
    \modulo{
      \mu_i \left( \left]a, b \right[ \right)
      -
      l_i(a+) \cdot \left( u(b-) - u(a+) \right)
    } \,.
  \end{eqnarray*}
  The first term in the right hand side is bounded by~\Ref{eq:Taylor}.
  By~\Ref{eq:mu}, for $I=\left] a, b \right[$.
  \begin{eqnarray*}
    \mu_i (I)
    & = &
    \int_I 
    \left( l_i \left( u(\xi) \right) - l_i \left(u(a+) \right) \right)
    \cdot 
    d\mu_c(\xi)
    \\
    & & 
    +
    \sum_{\xi \in I} 
    \left( 
      E_i \left( u(\xi-), u(\xi+) \right)
      -
      l_i \left( u(\xi-) \right)
      \cdot 
      \left( u(\xi+) - u(\xi-) \right)
    \right)
    \\
    & &
    +
    \sum_{\xi \in I}
    \left(
      l_i \left(u(\xi-) \right) - l_i \left( u(a+) \right)
    \right)
    \cdot
    \left( u(\xi+) - u(\xi-) \right)
    \\
    & &
    +
    l_i\left( u(a+) \right)
    \cdot
    \left(
      \sum_{\xi \in I} \left( u(\xi+) - u(\xi-) \right)
      +
      \int_I d\mu_c(\xi)
    \right) \,.
  \end{eqnarray*}
  We now estimates the different summands above separately. The
  Lipschitzeanity of $l_i$ ensures that the first summand above is
  bounded by
  \begin{displaymath}
    \modulo{
      \int_I
      \left( l_i \left( u(\xi) \right) - l_i \left(u(a+) \right) \right)
      \cdot 
      d\mu_c(\xi)
    }
    \leq 
    C 
    \cdot
    \diam \left( u (I) \right)
    \cdot
    \modulo{\mu} (I) \,.
  \end{displaymath}
  Passing to the second summand, using~\Ref{eq:Taylor}
  \begin{eqnarray*}
    & &
    \modulo{\sum_{\xi \in I} 
      \left( 
        E_i \left( u(\xi-), u(\xi+) \right)
        -
        l_i \left( u(\xi-) \right)
        \cdot 
        \left( u(\xi+) - u(\xi-) \right)
      \right)}
    \\
    & \leq &
    C \sum_{\xi \in I} \modulo{u(\xi+) - u(\xi-)}^2
    \\
    & \leq &
    C 
    \cdot
    \diam \left( u (I) \right)
    \cdot
    \modulo{\mu} (I) \,.
  \end{eqnarray*}
  Using again the Lipschitzeanity of $l_i$, the third summand is
  estimated as
  \begin{displaymath}
    \modulo{
      \sum_{\xi \in I}
      \left(
        l_i \left(u(\xi-) \right) - l_i \left( u(a+) \right)
      \right)
      \cdot
      \left( u(\xi+) - u(\xi-) \right)
    }
    \leq
    C 
    \cdot
    \diam \left( u (I) \right)
    \cdot
    \modulo{\mu} (I)
  \end{displaymath}
  while the last one can be rewritten as
  \begin{eqnarray*}
    & &
    l_i\left( u(a+) \right)
    \cdot
    \left(
      \sum_{\xi \in I} \left( u(\xi+) - u(\xi-) \right)
      +
      \int_I d\mu_c(\xi)
    \right)
    \\
    & = &
    l_i\left( u(a+) \right) \cdot
    \left( \mu_a(I) + \mu_c(I) \right)
    \\
    & = &
    l_i\left( u(a+) \right) \cdot \mu (I)
    \\
    & = &
    l_i\left( u(a+) \right) \cdot
    \left( u(b-) - u(a+) \right)
  \end{eqnarray*}
  completing the proof.
\end{proof}

\begin{lemma}
  \label{lemma:approx}
  Let $u \in \mathcal{D}_\delta$ with wave measure $\mu_1, \ldots,
  \mu_n$. Then, there exists a sequence $v_\nu \in
  \mathcal{D}_\delta^*$ with wave measures $\mu^\nu_1, \ldots,
  \mu^\nu_n$ such that
  \begin{eqnarray}
    \nonumber
    & &
    \lim_{\nu \to +\infty} \norma{v_\nu - u}_{\mathbf{L}^\infty}
    =
    0\,,
    \qquad
    \lim_{\nu \to +\infty} \norma{v_\nu - u}_{\L1}
    =
    0
    \\
    & &\label{eq:mu-convergence}
    \lim_{\nu \to +\infty} \mu_i^{\nu,\pm} (I)
    =
    \mu_i^{\pm} (I) \ 
    \mbox{ for any interval } I \subseteq \reali
    \\
    & &\label{eq:qconvergence}
    \lim_{\nu \to +\infty} \hat\mathbf{Q}(v_\nu)
    =
    \hat\mathbf{Q}(u)\,.
  \end{eqnarray}
  Moreover, an explicit definition of such a sequence is~\Ref{eq:vnu}.
\end{lemma}

\begin{proof}
  For $\nu \in \naturali \setminus \{0\}$ choose $a,b \in \reali$ with
  $b - a > 1$ and so that
  \begin{eqnarray*}
    \modulo{\mu} \left( \left] -\infty, a \right] \right)
    +
    \modulo{\mu} \left( \left[ b, +\infty \right[ \right)
    & \leq &
    \frac{1}{\nu}
    \\
    \int_{-\infty}^a \norma{u(x)} \, dx
    +
    \int_b^{+\infty} \norma{u(x)} \, dx
    & \leq &
    \frac{1}{\nu}
  \end{eqnarray*}
  Choose a finite sequence of real numbers $x_1, \ldots, x_N$ so that
  $a = x_1 < x_2 < \cdots < x_{N-1} < x_N = b$ and $\modulo{\mu}
  \left( \left] x_{\alpha-1}, x_\alpha \right[ \right) \leq
  \frac{1}{(b-a)\nu}$.
  Introduce the points $y_0 = x_1 - 1$, $y_\alpha = (x_\alpha +
  x_{\alpha+1}) / 2$ for $\alpha = 1, \ldots, N-1$ and $y_N = x_N +1$.
  Let
  \begin{equation}
    \label{eq:vnu}
    v_\nu
    =
    \sum_{\alpha=1}^N
    \left(
      u(x_\alpha-) \, \caratt{\left[y_{\alpha-1}, x_\alpha\right[}
      +
      u(x_\alpha+) \, \caratt{\left[x_\alpha, y_\alpha\right[} 
    \right) \,.
  \end{equation}

  Due to the above definitions, the $\L1$ and $\L\infty$ convergence
  $v_\nu \to u$ is immediate (observe that both $v_\nu$ and $u$ are
  right continuous).

  We first consider the intervals $I = \left]- \infty, x \right]$ for
  $x \in \reali$ or $I = \reali$.  Let $\mu_i^\nu$ be the wave measure
  corresponding to $v_\nu$. For notational simplicity, below we set
  $x_0 = -\infty$, $x_{N+1} = +\infty$, $u(x_0+) = 0$ and $u(x_{N+1}-)
  = 0$.  Let $\bar\alpha$ be such that $x \in \left[ x_{\bar\alpha},
    x_{\bar\alpha+1} \right[\ $ ($\bar \alpha = N$ for $I=\reali$).
  For any $i=1, \ldots, n$,
  \begin{eqnarray}
    \label{eq:inizio}
    \modulo{\mu_i^{\nu}} (I) - \modulo{\mu_i}(I)
    & = &
    \sum_{\alpha = 1}^{\bar\alpha}
    \modulo{E_i \left( u(x_\alpha-), u(x_\alpha+) \right)}
    \\
    \nonumber
    & &
    +
    \sum_{\alpha = 1}^{\bar\alpha}
    \modulo{E_i \left( u(x_{\alpha-1}+), u(x_{\alpha}-) \right)}
    \\
    \nonumber
    & &
    +
    \modulo{E_i \left(u(x_{\bar\alpha}+), u(x_{\bar\alpha+1}-)\right)}
    \caratt{\left[y_{\bar\alpha}, x_{\bar\alpha+1} \right[} (x)
    \\
    \nonumber
    & &
    -
    \sum_{\alpha = 1}^{\bar\alpha}
    \modulo{\mu_i} \left( \{x_\alpha\}\right)
    -
    \sum_{\alpha = 1}^{\bar\alpha}
    \modulo{\mu_i} \left( \left] x_{\alpha-1}, x_{\alpha} \right[ \right)
    \\
    \nonumber
    & &
    -
    \modulo{\mu_i} \left( \left] x_{\bar\alpha}, x \right] \right) \,.
  \end{eqnarray}
  Observe that
  \begin{eqnarray*}
    E_i \left( u(x_\alpha-), u(x_\alpha+) \right)
    & = &
    \mu_i \left( \{x_\alpha\}\right)
    \quad\qquad \mbox{ for all } \quad
    \alpha = 1, \ldots, N
    \\
    \modulo{E_i \left( u(x_\alpha-), u(x_\alpha+) \right)}
    & = &
    \modulo{\mu_i} \left( \{x_\alpha\}\right)
    \ \qquad \mbox{ for all } \quad
    \alpha = 1, \ldots, N
    \\
    \modulo{E_i \left(u(x_{\bar\alpha}+), u(x_{\bar\alpha+1}-)\right)}
    & \leq &
    C \, 
    \modulo{\mu} \left( \left] x_{\bar\alpha}, x_{\bar\alpha+1} \right[ \right)
    \; \leq \;
    C /\nu
    \\
    \modulo{\mu_i} \left( \left] x_{\bar\alpha}, x \right] \right)
    & \leq &
    C \, \modulo{\mu}
    \left( \left] x_{\bar\alpha}, x_{\bar\alpha+1} \right[ \right)
    \; \leq \;
    C /\nu
  \end{eqnarray*}
  so that, by Lemma~\ref{lem:diam}
  \begin{eqnarray}
    \nonumber
    \modulo{\mu_i^{\nu}} (I) - \modulo{\mu_i}(I) \!\!\!
    & = &
    \!
    \sum_{\alpha=1}^{\bar\alpha} \!
    \left(
      \modulo{E_i \! \left( u(x_{\alpha-1}+), u(x_{\alpha}-) \right)}
      -
      \modulo{\mu_i} \left( \left] x_{\alpha-1}, x_\alpha \right[ \right)
    \right)
    + \frac{C}{\nu}
    \\
    \nonumber
    & \leq &
    \!
    \sum_{\alpha=1}^{\bar\alpha} \!
    \left(
      \modulo{E_i \! \left( u(x_{\alpha-1}+), u(x_{\alpha}-) \right)}
      -
      \modulo{\mu_i \left( \left] x_{\alpha-1}, x_\alpha \right[ \right)}
    \right)
    + \frac{C}{\nu}
    \\
    \nonumber
    & \leq &
    \sum_{\alpha=1}^{\bar\alpha}
    \modulo{
      E_i \left( u(x_{\alpha-1}+), u(x_{\alpha}-) \right) 
      -
      \mu_i \left( \left] x_{\alpha-1}, x_\alpha \right[ \right)
    }
    + \frac{C}{\nu}
    \\
    \nonumber
    & \leq &
    C \cdot \sum_{\alpha=1}^{\bar\alpha}
    \diam \left( u\left( \left]x_{\alpha-1}, x_\alpha\right[ \right) \right)
    \modulo{\mu} \left( \left] x_{\alpha-1}, x_\alpha \right[ \right)
    + \frac{C}{\nu}
    \\
    \nonumber
    & \leq &
    C \cdot \sum_{\alpha=1}^{\bar\alpha}
    \frac{1}{\nu}
    \modulo{\mu} \left( \left] x_{\alpha-1}, x_\alpha \right[ \right)
    + \frac{C}{\nu}
    \\
    \label{eq:fine}
    & \leq &
    C \cdot \left( 1+\modulo{\mu} (\reali) \right) \cdot \frac{1}{\nu}
  \end{eqnarray}
  and $\displaystyle \limsup_{\nu \to +\infty} \modulo{\mu_i^\nu} (I)
  \leq \modulo{\mu_i} (I)$.

  Passing to the other inequality, introduce the functions
  \begin{displaymath}
    w_i(x) = \mu_i \left( \left] -\infty, x\right] \right)
    \quad \mbox{ and } \quad
    w_i^\nu(x) = \mu_i^\nu \left( \left] -\infty, x\right] \right) \,.
  \end{displaymath}
  and repeat the same computations used
  in~\Ref{eq:inizio}--\Ref{eq:fine} to obtain
  \begin{displaymath}
    \modulo{w_i^\nu(x) - w_i(x)}
    =
    \modulo{\mu_i^\nu(I) - \mu_i(I)}
    \leq
    C \cdot \left( 1+\modulo{\mu} (\reali) \right) \cdot \frac{1}{\nu} 
  \end{displaymath}
  showing that $w_i^\nu \to w_i$ uniformly on $\reali$. By the lower
  semicontinuity of the total variation
  \begin{displaymath}
    \modulo{\mu_i} (I) 
    =
    \tv \left( w_i, \left]-\infty, x\right] \right)
    \leq
    \liminf_{\nu \to +\infty}
    \tv \left( w_i^\nu, \left]-\infty, x\right] \right)
    =
    \liminf_{\nu \to +\infty}
    \modulo{\mu_i^\nu} (I)
  \end{displaymath}
  showing that $\modulo{\mu_i^\nu}(I) \to \modulo{\mu_i}(I)$ as $\nu
  \to +\infty$.  This convergence, together with the uniform
  convergence above, imply that
  \begin{eqnarray*}
    \lim_{\nu \to +\infty}
    \left( \mu_i^{\nu,+} (I) + \mu_i^{\nu,-} (I) \right)
    & = &
    \mu_i^{+} (I) + \mu_i^{-} (I)
    \\
    \lim_{\nu \to +\infty}
    \left( \mu_i^{\nu,+} (I) - \mu_i^{\nu,-} (I) \right)
    & = &
    \mu_i^{+} (I) - \mu_i^{-} (I)
  \end{eqnarray*}
  which together imply~\Ref{eq:mu-convergence} for $I =
  \left]-\infty,x \right]$ or $I = \reali$.  Let now $\tilde u(x)$ and
  $\tilde v_\nu(x)$ be the right continuous representative of
  respectively $u(-x)$ and $v_\nu(-x)$ with corresponding wave
  measures $\tilde\mu_i$ and $\tilde\mu_i^\nu$. The previous
  computation show that $\tilde\mu_i^{\nu,\pm} \left( \left]-\infty,x\right]
  \right) \to \tilde\mu_i^\pm \left(\left]-\infty,x\right] \right)$ for
  all real $x$. But $\tilde\mu_i^{\nu,\pm} \left( \left]-\infty,x \right]
  \right) = \mu_i^{\nu,\pm} \left( \left[-x,+\infty\right[ \right)$ and
  $\tilde\mu_i^\pm \left(\left]-\infty,x \right] \right) = \mu_i^\pm\left(
    \left[-x,+\infty \right[ \right)$. Therefore,
  \Ref{eq:mu-convergence} holds also for the intervals $I =
  \left[x,+\infty \right[$ and therefore for all real intervals $I
  \subseteq \reali$.

  Passing to~\Ref{eq:qconvergence} we observe that it is enough to
  show the convergence of every single term in the sum~\Ref{eq:defrho}
  which defines the measure $\rho$. Since the computations for these
  terms are identical, we show the convergence of only one, say
  $\mu_i^{\nu,+}\otimes\mu_i^{\nu,-}$.  Fix $\epsilon>0$ and choose a
  finite set of real numbers $x_1, \ldots, x_{N_\epsilon}$ so
  that $-\infty=x_0^\epsilon < x_1^\epsilon < \cdots <
  x_{N_\epsilon}^\epsilon < x_{N_\epsilon+1}^\epsilon=+\infty$ and
  $\modulo{\mu} \left( \left] x_{\alpha-1}^\epsilon, x_\alpha^\epsilon
    \right[ \right) \leq \epsilon$. To simplify the notations, we
  define
  \begin{displaymath}
    K 
    = 
    \left\{ (x,y) \in \reali^2 \colon x < y \right\}
    ,\quad  
    \tau^\nu
    =
    \mu_i^{\nu,+} \otimes \mu_i^{\nu,-}
    \mbox{ and} \quad
    \tau = \mu_i^{+} \otimes \mu_i^{-}. 
  \end{displaymath}
  Now, write $K$ as the union of a finite family of disjoint sets as:
  \begin{displaymath}
    R_{\alpha,\beta}
    =
    \left[x_\alpha^\epsilon,x_{\alpha+1}^\epsilon \right[
    \times
    \left[x_\beta^\epsilon,x_{\beta +1}^\epsilon \right[,
    \quad
    K
    =
    \bigg(
    \bigcup_{{\alpha,\beta=0 \atop\alpha<\beta}}^{ N_\epsilon}
    R_{\alpha,\beta} 
    \bigg)
    \cup
    \bigcup_{\alpha=0}^{N_\epsilon} 
    \left(R_{\alpha,\alpha} \cap K\right)
  \end{displaymath}
  and compute
  \begin{eqnarray*}
    \modulo{\tau^\nu(K)-\tau(K)}
    & \leq &
    \sum_{{\alpha,\beta=0 \atop\alpha<\beta}}^{N_\epsilon} 
    \modulo{\tau^\nu(R_{\alpha,\beta}) - \tau(R_{\alpha,\beta})}
    \\
    & &
    \quad +
    \sum_{\alpha=0}^{N_\epsilon} 
    \left(
      \tau^\nu (R_{\alpha,\alpha} \cap K)
      +
      \tau (R_{\alpha,\alpha}\cap K)
    \right).
  \end{eqnarray*}
  We have also the estimate
  \begin{displaymath}
    \begin{array}{rcccl}
      \tau^\nu (R_{\alpha,\alpha} \cap K)
      \!\! & \leq & \!\!
      \tau^\nu \!
      \left(
        R_{\alpha,\alpha}
        \setminus 
        \left\{ (x_\alpha^\epsilon,x_\alpha^\epsilon ) \right\} 
      \right)
      \! & = & \!
      \tau^\nu ( R_{\alpha, \alpha})
      -
      \tau^\nu \!
      \left( \left\{ (x_\alpha^\epsilon,x_\alpha^\epsilon) \right\} \right)
      \\
      \tau (R_{\alpha,\alpha} \cap K)
      \!\! & \leq & \!\!
      \tau 
      \left(
        R_{\alpha,\alpha} 
        \setminus 
        \left\{ (x_\alpha^\epsilon,x_\alpha^\epsilon) \right\}
      \right)
      \! & = & \!
      \tau (R_{\alpha,\alpha})
      -
      \tau \left( \left\{(x_\alpha^\epsilon,x_\alpha^\epsilon) \right\} \right)
    \end{array}
  \end{displaymath}  
  The limit~\Ref{eq:mu-convergence} implies that the product measure
  converges on rectangles:
  \begin{eqnarray*}
    \lim_{\nu\to+\infty}\tau^\nu(R_{\alpha,\beta})
    & = &
    \tau \left(R_{\alpha,\beta}\right)
    \\
    \lim_{\nu\to+\infty} 
    \tau^\nu \left( \left\{ \left(x_\alpha^\epsilon, x_\alpha^\epsilon
        \right) \right\} \right)
    & = &
    \tau
    \left(\left\{ \left(x_\alpha,x_\alpha \right) \right\} \right) \,.
  \end{eqnarray*}
  Therefore
  \begin{eqnarray*}
    \limsup_{\nu\to +\infty} \modulo{\tau^\nu(K)-\tau(K)}
    & \leq & 
    2\sum_{\alpha=0}^{N_\epsilon} 
    \left(
      \tau (R_{\alpha,\alpha}) 
      - 
      \tau \left(\left\{(x_\alpha^\epsilon,x_\alpha^\epsilon)
        \right\}\right)
    \right)
    \\
    & = &
    2\sum_{\alpha=0}^{N_\epsilon} 
    \bigg(
    \mu_i^
    +
    \left( \left\{x^\epsilon_\alpha\right\} \right)
    \mu_i^- \left( ]x^\epsilon_\alpha,x^\epsilon_{\alpha+1}[ \right)
    \\
    & &
    +
    \mu_i^+ \left( ]x^\epsilon_\alpha,x^\epsilon_{\alpha+1}[\right)
    \mu_i^- \left( \{x^\epsilon_\alpha\} \right)
    \\
    & &
    +
    \mu_i^+ \left( ]x^\epsilon_\alpha,x^\epsilon_{\alpha+1}[\right) 
    \mu_i^- \left( ]x^\epsilon_\alpha,x^\epsilon_{\alpha+1}[\right)
    \bigg)
    \\
    & \leq & 
    C \sum_{\alpha=0}^{N_\epsilon} \epsilon 
    \modulo{\mu} \left( [x^\epsilon_\alpha,x^\epsilon_{\alpha+1}[ \right)
    \\
    & \leq & 
    C \epsilon \modulo{\mu} (\reali) \,.
  \end{eqnarray*}
  By the arbitrariness of $\epsilon$, $\displaystyle \lim_{\nu\to
    +\infty}\tau^\nu (K) = \tau (K)$, completing the proof.
\end{proof}

The following proof, simpler than that
in~\cite[Theorem~10.1]{BressanLectureNotes}, is based on the piecewise
constant approximations introduced in Lemma~\ref{lemma:approx}.

\begin{proposition}
  \label{prop:lowsec}
  The functionals $\hat \mathbf{Q}$ and $\hat\mathbf{\Upsilon}$
  defined in~\Ref{eq:def2Q} and~\Ref{eq:def2U} are lower
  semicontinuous with respect to the $\mathbf{L}^1$ norm.
\end{proposition}

\begin{proof}
  We prove the lower semicontinuity only of $\hat\mathbf{\Upsilon}$
  since the other one is easier. Take $u_\nu,\;u \in
  \mathcal{D}_\delta$ such that $u_\nu \to u$ in $\L1$. By
  Lemma~\ref{lemma:approx}, there exists a sequence of piecewise
  constant functions $v_\nu \in \mathcal{D}_\delta^*$ such that
  \begin{displaymath}
    \norma{v_\nu-u_\nu}_{\L1} \leq \frac{1}{\nu}
    \,,\qquad
    \hat\mathbf{\Upsilon} (v_\nu) 
    \leq 
    \hat\mathbf{\Upsilon}(u_\nu) + \frac{1}{\nu} \,.
  \end{displaymath}
  If we define $l = \liminf_{\nu\to +\infty}
  \hat\mathbf{\Upsilon}(v_\nu)$, then $l \leq \liminf_{\nu\to +\infty}
  \hat\mathbf{\Upsilon}(u_\nu)$. Moreover, by possibly passing to a
  subsequence, we suppose that $l = \lim_{\nu\to +\infty}
  \hat\mathbf{\Upsilon}(v_\nu)$ and that $v_\nu(x)\to u(x)$ for every
  $x\in D$, where $\reali\setminus D$ has zero Lebesgue measure. Fix
  now $\epsilon>0$ and, using again Lemma~\ref{lemma:approx}, choose a
  piecewise constant function $v_\epsilon\in\mathcal{D}_\delta^*$
  which approximates $u$:
  \begin{displaymath}
    v_\epsilon
    =
    \sum_{\alpha=1}^{N_\epsilon}
    \left(
      u(x_\alpha^\epsilon-) \, 
      \caratt{\left[y_{\alpha-1}^\epsilon, 
          x_\alpha^\epsilon\right[}
      +
      u(x_\alpha^\epsilon+) \, \caratt{\left[x_\alpha^\epsilon, 
          y_\alpha^\epsilon\right[} 
    \right)
  \end{displaymath}
  \begin{displaymath}
    \hat\mathbf{\Upsilon}(u)
    \leq 
    \hat\mathbf{\Upsilon}\left(v_\epsilon\right) + \epsilon 
    =
    \mathbf{\Upsilon}\left(v_\epsilon\right) + \epsilon
    \,.
  \end{displaymath}
  By Remark~\ref{rem:Lipschitz}, for every $\alpha=1,\ldots,
  N_\epsilon$, we can choose points
  \begin{displaymath}
    \check x_{\alpha}^{\epsilon-} \in 
    \left]y_{\alpha-1}^\epsilon, x_\alpha^\epsilon \right[ \cap D
    \quad\mbox{ and }\quad
    \check x_{\alpha}^{\epsilon+} \in 
    \left] x_{\alpha}^{\epsilon},y_\alpha^\epsilon \right[ \cap D 
  \end{displaymath}
  such that the function
  \begin{displaymath}
    \bar v_\epsilon
    =
    \sum_{\alpha=1}^{N_\epsilon}
    \left(
      u(\check x_{\alpha}^{\epsilon-}) \, 
      \caratt{\left[y_{\alpha-1}^\epsilon, 
          x_\alpha^\epsilon\right[}
      +
      u(\check x_{\alpha}^{\epsilon+}) \, 
      \caratt{\left[x_\alpha^\epsilon, 
          y_\alpha^\epsilon\right[} 
    \right)
  \end{displaymath}
  satisfies $ \mathbf{\Upsilon}(v_\epsilon) \leq
  \mathbf{\Upsilon}\left(\bar v_\epsilon\right) + \epsilon$. Define
  now
  \begin{displaymath}
    \check v_{\epsilon,\nu}
    =
    \sum_{\alpha=1}^{N_\epsilon}
    \left(
      v_\nu(\check x_{\alpha}^{\epsilon-}) \, 
      \caratt{\left[y_{\alpha-1}^\epsilon, 
          x_\alpha^\epsilon\right[}
      +
      v_\nu(\check x_{\alpha}^{\epsilon+}) \, 
      \caratt{\left[x_\alpha^\epsilon, 
          y_\alpha^\epsilon\right[} 
    \right) \,.
  \end{displaymath}
  Since $v_\nu\left(\check x_{\alpha}^{\epsilon\pm}\right)\to
  u\left(\check x_{\alpha}^{\epsilon\pm}\right)$, we can apply again
  Remark~\ref{rem:Lipschitz} to obtain that for $\nu$ sufficiently
  large one has $\mathbf{\Upsilon}\left(\bar v_\epsilon\right)\le
  \mathbf{\Upsilon}\left(\check v_{\epsilon,\nu}\right) +\epsilon$.

  But $\check v_{\epsilon,\nu}$ is obtained by removing an ordered
  sequence of values attained by $v_\nu$, therefore we can apply
  Proposition~\ref{prop:reduce} to get $\mathbf{\Upsilon}\left(\check
    v_{\epsilon,\nu}\right)\leq \mathbf{\Upsilon}\left(v_\nu\right)$
  and therefore we have the following chain of inequalities:
  \begin{displaymath}
    \hat\mathbf{\Upsilon}\left(u\right)
    \leq 
    \mathbf{\Upsilon}\left(v_\epsilon\right)
    +
    \epsilon
    \leq 
    \mathbf{\Upsilon}\left( \bar v_\epsilon\right)
    +
    2\epsilon
    \leq 
    \mathbf{\Upsilon}\left(\check v_{\epsilon,\nu}\right)
    +
    3 \epsilon
    \leq
    \mathbf{\Upsilon}\left(v_\nu\right)
    +
    3\epsilon \,.
  \end{displaymath}
  Taking the limit as $\nu\to + \infty$ one obtains
  $\hat\mathbf{\Upsilon}\left(u\right) \leq l+3\epsilon$ and the
  arbitrariness of $\epsilon>0$ implies $\hat\mathbf{\Upsilon}
  \left(u\right) \leq l \leq \liminf_{\nu\to +\infty}
  \hat\mathbf{\Upsilon}(u_\nu)$.
\end{proof}

\begin{corollary}
  \label{prop:equal}
  The functionals $\hat\mathbf{Q}$ and $\hat\mathbf{\Upsilon}$ defined
  by ~\Ref{eq:def2Q} and~\Ref{eq:def2U} coincide on all
  $\mathcal{D}_\delta$ with $\mathbf{Q}$ and $\mathbf{\Upsilon}$.
\end{corollary}

\begin{proof}
  It is a straightforward consequence of the lower semicontinuity of
  both $\hat\mathbf{Q}$ and $\hat\mathbf{\Upsilon}$
  (Proposition~\ref{prop:lowsec}) and $\mathbf{Q}$ and
  $\mathbf{\Upsilon}$ (Proposition~\ref{prop:Qlsc}).
  
  Indeed, consider only $\hat\mathbf{\Upsilon}$ and
  $\mathbf{\Upsilon}$. They obviously coincide on
  $\mathcal{D}_\delta^*$. Fix now $u\in\mathcal{D}_\delta$. By the
  definition~\Ref{eq:Q} of $\mathbf{\Upsilon}$, there exists a
  sequence $v_\nu$ of functions in $\mathcal{D}_\delta^*$ converging
  to $u$ in $\L1$ and such that $\mathbf{\Upsilon} (v_\nu) \to
  \mathbf{\Upsilon}(u)$ as $\nu \to +\infty$. By the lower
  semicontinuity of $\hat\mathbf{\Upsilon}$
  (Proposition~\ref{prop:lowsec}), we obtain
  \begin{displaymath}
    \hat\mathbf{\Upsilon} (u) 
    \leq 
    \liminf_{\nu\to+\infty} \hat\mathbf{\Upsilon} (v_\nu) 
    =
    \liminf_{\nu\to+\infty} \mathbf{\Upsilon}(v_\nu)
    =
    \lim_{\nu\to+\infty} \mathbf{\Upsilon}(v_\nu)
    =
    \mathbf{\Upsilon}(u) \,.
  \end{displaymath}
  Analogously, by Lemma~\ref{lemma:approx} we can take a sequence
  $v_\nu$ of functions in $\mathcal{D}_\delta^*$ such that $v_\nu \to
  u$ in $\L1$ and $\hat\mathbf{\Upsilon}(v_\nu) \to
  \hat\mathbf{\Upsilon}(u)$ as $\nu \to +\infty$.  Therefore, along
  this particular sequence, we may repeat the estimates as above
  applying the lower semicontinuity of $\mathbf{\Upsilon}$
  (Proposition~\ref{prop:Qlsc}):
  \begin{displaymath}
    \mathbf{\Upsilon} (u) 
    \leq 
    \liminf_{\nu\to+\infty} \mathbf{\Upsilon} (v_\nu) 
    =
    \liminf_{\nu\to+\infty} \hat\mathbf{\Upsilon} (v_\nu)
    =
    \lim_{\nu\to+\infty} \hat\mathbf{\Upsilon} (v_\nu)
    =
    \hat\mathbf{\Upsilon}(u).
  \end{displaymath}
\end{proof}

\noindent By Corollary~\ref{prop:equal}, in the following we write
$\mathbf{Q}$ and $\mathbf{\Upsilon}$ for $\hat\mathbf{Q}$ and
$\hat\mathbf{\Upsilon}$.

The following proposition shows the lower semicontinuity of $\hat
\mathbf{\Xi}$ along piecewise constant converging sequences.

\begin{lemma}
  \label{lemma:approx2}
  Let $u,\;\tilde u\in\mathcal{D}_\delta$. Then, the approximating
  sequences $v_\nu,\;\tilde v_\nu\in\mathcal{D}_\delta^*$ defined in
  Lemma~\ref{lemma:approx} satisfy also $\lim_{\nu\to +\infty}
  \hat\mathbf{\Xi} (v_\nu,\tilde v_\nu) = \hat\mathbf{\Xi} (u,\tilde
  u)$.
\end{lemma}

\begin{proof}
  Define $\mathbf{q}(x)$ and $\mathbf{q}^\nu(x)$ so that $\tilde u(x)
  = \mathbf{S} \left(\mathbf{q}(x)\right) \left(u(x)\right)$ and
  $\tilde v_\nu(x) = \mathbf{S}\left(\mathbf{q}^\nu(x)\right)
  \left(v_\nu(x)\right)$. Then, $\mathbf{q}^\nu \to \mathbf{q}$
  uniformly and in $\L1$. Let $\mathbf{W}_i(x)$, respectively
  $\mathbf{W}_i^\nu(x)$, be the weights defined in~\Ref{def:barw} with
  reference to $u, \tilde u$, respectively $v_\nu, \tilde v_\nu$.
  Compute
  \begin{eqnarray*}
    \modulo{
      \hat\mathbf{\Xi} \left(v_\nu,\tilde v_\nu \right)-
      \hat\mathbf{\Xi}\left(u,\tilde u\right)
    }
    & \leq &
    \sum_{i=1}^n
    \int_{-\infty}^{+\infty}
    \modulo{q_i^\nu(x)-q_i(x)}
    \cdot
    \mathbf{W}_i^\nu(x) \, dx
    \\
    & &
    + 
    \sum_{i=1}^n
    \int_{-\infty}^{+\infty}
    \modulo{q_i(x)}
    \cdot
    \modulo{\mathbf{W}_i^\nu(x) - \mathbf{W}_i(x)} \, dx
  \end{eqnarray*}
  The first integral converges obviously to zero. Concerning the
  second one, where $q_i(x) = 0$ the integrand vanishes. Otherwise, if
  $q_i(x) \neq 0$, then for $\nu$ sufficiently large $q_i(x) \cdot
  q_i^\nu(x) > 0$ and hence the weights depend continuously only on
  the wave measures and on the interaction potentials which all
  converge by Lemma~\ref{lemma:approx}. Therefore, for all $x \in
  \reali$ the integrand satisfies
  \begin{displaymath}
    \lim_{\nu\to +\infty} 
    \modulo{q_i(x)} \cdot \modulo{\mathbf{W}_i^\nu(x) - \mathbf{W}_i(x)} 
    =0 \,.
  \end{displaymath}
  The Dominated Convergence Theorem concludes the proof.
\end{proof}

\begin{proposition}
  \label{prop:lscfunctional}
  For any $u,\;\tilde u \in \mathcal{D}_\delta$ and $v_\nu,\;\tilde
  v_\nu \in \mathcal{D}_\delta^*$ such that $v_\nu \to u$ and $\tilde
  v_\nu \to \tilde u$ in $\L1$ as $\nu\to +\infty$, we have
  $\hat\mathbf{\Xi} (u,\tilde u) \leq \liminf_{\nu\to +\infty}
  \hat\mathbf{\Xi} (v_\nu,\tilde v_\nu)$.
\end{proposition}

\begin{proof}
  Let $l = \liminf_{\nu\to +\infty}\hat \mathbf{\Xi} (v_\nu,\tilde
  v_\nu)$. Passing to subsequences, we assume that $l = \lim_{\nu\to
    +\infty}\hat\mathbf{\Xi}\left(v_\nu,\tilde v_\nu\right)$ and that
  $v_\nu$, respectively $\tilde v_\nu$, converges pointwise to $u$,
  respectively $\tilde u$, on a set $D \subseteq \reali$ with $\reali
  \setminus D$ having zero Lebesgue measure.  Fix $\epsilon > 0$ and
  apply lemmas~\ref{lemma:approx} and~\ref{lemma:approx2} to find two
  functions
  \begin{displaymath}
    v_\epsilon
    =
    \sum_{\alpha=1}^{N_\epsilon}
    \left(
      u(x_\alpha^\epsilon-) \, 
      \caratt{\left[y_{\alpha-1}^\epsilon, 
          x_\alpha^\epsilon\right[}
      +
      u(x_\alpha^\epsilon+) \, \caratt{\left[x_\alpha^\epsilon, 
          y_\alpha^\epsilon\right[} 
    \right) \,,
  \end{displaymath}
  \begin{displaymath}
    \tilde v_\epsilon
    =
    \sum_{\alpha=1}^{\tilde N_\epsilon}
    \left(
      \tilde u(\tilde x_\alpha^\epsilon-) \, 
      \caratt{\left[\tilde y_{\alpha-1}^\epsilon, 
          \tilde x_\alpha^\epsilon\right[}
      +
      \tilde u(\tilde x_\alpha^\epsilon+) \, \caratt{\left[\tilde x_\alpha^\epsilon, 
          \tilde y_\alpha^\epsilon\right[} 
    \right) \,,
  \end{displaymath}
  such that $\modulo{\hat\mathbf{\Xi} (u,\tilde u) - \hat\mathbf{\Xi}
    (v_\epsilon,\tilde v_\epsilon)} + \norma{u-v_\epsilon}_{\L1} +
  \norma{\tilde u-\tilde v_\epsilon}_{\L1} \leq \epsilon$.  On
  $\mathcal{D}_\delta^*$, $\hat\mathbf{\Xi}$ and $\gf$ coincide, hence
  Remark~\ref{rem:ContPhi} applies and there are points
  \begin{eqnarray*}
    \check x_{\alpha}^{\epsilon-}
    \in 
    \left] y_{\alpha-1}^\epsilon, x_\alpha^\epsilon \right[ \cap D
    \quad
    &\mbox{ and }&
    \quad
    \check x_{\alpha}^{\epsilon+}
    \in 
    \left] x_{\alpha}^{\epsilon}, y_\alpha^\epsilon\right[ \cap D
    \\ 
    \check{\tilde x}_{\alpha}^{\epsilon-} \in 
    \left] \tilde y_{\alpha-1}^\epsilon, \tilde x_\alpha^\epsilon\right[ \cap D
    \quad
    &\mbox{ and }&
    \quad
    \check{\tilde x}_{\alpha}^{\epsilon+} \in 
    \left] \tilde x_{\alpha}^{\epsilon}, \tilde y_\alpha^\epsilon\right[ \cap D 
  \end{eqnarray*}
  such that the two functions
  \begin{eqnarray*}
    \check v_{\epsilon,\nu}
    & = &
    \sum_{\alpha=1}^{N_\epsilon}
    \left(
      v_\nu(\check x_\alpha^{\epsilon-}) \, 
      \caratt{\left[y_{\alpha-1}^\epsilon, 
          x_\alpha^\epsilon\right[}
      +
      v_\nu(\check x_\alpha^{\epsilon+}) \, \caratt{\left[x_\alpha^\epsilon, 
          y_\alpha^\epsilon\right[} 
    \right)
    \\
    \check{ \tilde v}_{\epsilon,\nu}
    & = &
    \sum_{\alpha=1}^{\tilde N_\epsilon}
    \left(
      \tilde v_\nu(\check{\tilde x}_\alpha^{\epsilon-}) \, 
      \caratt{\left[\tilde y_{\alpha-1}^\epsilon, 
          \tilde x_\alpha^\epsilon\right[}
      +
      \tilde v_\nu(\check{\tilde x}_\alpha^{\epsilon+}) \, \caratt{\left[\tilde x_\alpha^\epsilon, 
          \tilde y_\alpha^\epsilon\right[} 
    \right)
  \end{eqnarray*}
  satisfy $\modulo{\hat\mathbf{\Xi} (v_\epsilon,\tilde v_\epsilon) -
    \hat\mathbf{\Xi} (\check v_{\epsilon,\nu},\check{\tilde
      v}_{\epsilon,\nu})} + \norma{\check v_{\epsilon,\nu} -
    v_\epsilon}_{\L1} +\norma{\check {\tilde v}_{\epsilon,\nu} -
    \tilde v_\epsilon}_{\L1} \leq \epsilon$ for $\nu$ sufficiently
  large. Lemma~\ref{lemma:reduce} thus implies
  \begin{displaymath}
    \hat \mathbf{\Xi} (\check v_{\epsilon,\nu},\check{\tilde v}_{\epsilon,\nu})
    \leq 
    \hat \mathbf{\Xi} ( v_{\nu},\tilde v_{\nu})
    +
    C \left( 
      \norma{\check v_{\epsilon,\nu} - v_{\nu}}_{\L1} +
      \norma{\check{\tilde v}_{\epsilon,\nu} - \tilde v_{\nu}}_{\L1}
    \right) \,.
  \end{displaymath}
  Hence
  \begin{eqnarray*}
    \hat\mathbf{\Xi} (u,\tilde u)
    & \leq &
    \hat \mathbf{\Xi} (v_\epsilon, \tilde v_\epsilon) + \epsilon
    \\
    & \leq &
    \hat\mathbf{\Xi} ( \check v_{\epsilon,\nu},\check{\tilde v}_{\epsilon,\nu}) +
    2 \epsilon
    \\
    & \leq &
    \hat\mathbf{\Xi} ( v_{\nu},\tilde v_{\nu}) +
    2 \epsilon +
    C \left(
      \norma{\check v_{\epsilon,\nu}- v_{\nu}}_{\L1} +
      \norma{\check{\tilde v}_{\epsilon,\nu}- \tilde v_{\nu}}_{\L1}
    \right)
    \\
    & \leq &
    \hat\mathbf{\Xi} ( v_{\nu},\tilde v_{\nu}) +
    2\epsilon \cdot (C+1) +
    C \left(
      \norma{v_{\epsilon}- v_{\nu}}_{\L1} +
      \norma{{\tilde v}_{\epsilon}- \tilde v_{\nu}}_{\L1}
    \right)
  \end{eqnarray*}
  so that, for $\nu\to +\infty$,
  \begin{displaymath}
    \hat\mathbf{\Xi} (u,\tilde u)
    \leq
    l+
    2\epsilon (C+1) +
    C \left(
      \norma{v_{\epsilon}- u}_{\L1} +
      \norma{\tilde v_{\epsilon}- \tilde u}_{\L1}
    \right)
    \leq
    l + 2\epsilon (2C+1) \,.
  \end{displaymath}
  The arbitrariness of $\epsilon$ concludes the proof.
\end{proof}

\begin{theorem}
  Let $f$ satisfy~\textbf{(F)}. The functional $\hat\mathbf{\Xi}$
  defined in~\Ref{eq:barPhi} coincides on all $\mathcal{D}_\delta$
  with $\mathbf{\Xi}$ as defined in~\Ref{eq:Xi}. In particular
  $\hat\mathbf{\Xi}$ is lower semicontinuous.
\end{theorem}

\begin{proof}
  Both functionals can be approximated through their evaluation on
  piecewise constant functions (see~\Ref{eq:Xi} and
  Lemma~\ref{lemma:approx2}). Both functionals coincide on piecewise
  constant functions and are lower semicontinuous along sequences
  of piecewise constant functions (see Theorem~\ref{thm:main} and
  Proposition~\ref{prop:lscfunctional}). A procedure
  identical to that of Corollary~\ref{prop:equal} completes the proof.
\end{proof}

\noindent \textbf{Acknowledgment:} We thank an anonymous referee for
suggesting to consider also the wave measure formulation.

{\small{

  }}

\end{document}